\documentclass[12pt]{article}
\usepackage{amsfonts}
\usepackage{amssymb}
\usepackage{fullpage}
\usepackage{xy}
\xyoption{all}

\begin{document}




\newtheorem{thm}{Theorem}[section]
\newtheorem{lem}[thm]{Lemma}
\newtheorem{prop}[thm]{Proposition}
\newtheorem{df}[thm]{Definition}
\newtheorem{cor}[thm]{Corollary}
\newtheorem{rem}[thm]{Remark}
\newtheorem{ex}[thm]{Example}
\newenvironment{proof}{\medskip
\noindent {\bf Proof.}}{\hfill \rule{.5em}{1em}\mbox{}\bigskip}

\def\eps{\varepsilon}

\def\GlR#1{\mbox{\it Gl}(#1,\R)}
\def\glR#1{{\frak gl}(#1,\R)}
\def\GlC#1{\mbox{\it Gl}(#1,\C)}
\def\glC#1{{\frak gl}(#1,\C)}
\def\Gl#1{\mbox{\it Gl}(#1)}

\def\ov{\overline}
\def\ot{\otimes}
\def\und{\underline}
\def\w{\wedge}
\def\ra{\rightarrow}
\def\lra{\longrightarrow}

\def\rk{\mbox{rk}}
\def\mod{\mbox{ mod }}

\newcommand{\al}{\alpha}
\newcommand{\be}{\beta}
\newcommand{\ga}{\gamma}
\newcommand{\la}{\lambda}
\newcommand{\om}{\omega}
\newcommand{\Om}{\Omega}
\renewcommand{\th}{\theta}
\newcommand{\Th}{\Theta}
\renewcommand{\phi}{\varphi}

\newcommand{\lb}{\langle}
\newcommand{\rb}{\rangle}

\def\pair#1#2{\left<#1,#2\right>}
\def\big#1{\displaystyle{#1}}

\def\Z{{\Bbb Z}}
\def\R{{\Bbb R}}
\def\C{{\Bbb C}}
\def\O{{\Bbb O}}
\def\CP{{\Bbb C} {\Bbb P}}
\def\HP{{\Bbb H} {\Bbb P}}
\def\P{{\Bbb P}}
\def\Q{{\Bbb Q}}
\def\H{{\Bbb H}}
\def\F{{\Bbb F}}

\def\FV{{\frak F}_V}

\def\so{{\frak {so}}}
\def\co{{\frak {co}}}
\def\sl{{\frak {sl}}}
\def\su{{\frak {su}}}
\def\sp{{\frak {sp}}}
\def\csp{{\frak {csp}}}
\def\spin{{\frak {spin}}}
\def\g{{\frak g}}
\def\h{{\frak h}}
\def\s{{\frak s}}
\def\k{{\frak k}}
\def\l{{\frak l}}
\def\m{{\frak m}}
\def\n{{\frak n}}
\def\t{{\frak t}}
\def\u{{\frak u}}
\def\z{{\frak z}}
\def\p{{\frak p}}
\def\L{{\frak L}}
\def\X{{\frak X}}
\def\gl{{\frak {gl}}}
\def\hol{{\frak {hol}}}
\renewcommand{\frak}{\mathfrak}
\renewcommand{\Bbb}{\mathbb}

\def\hook{\mbox{}\begin{picture}(10,10)\put(1,0){\line(1,0){7}}
  \put(8,0){\line(0,1){7}}\end{picture}\mbox{}}

\def\be{\begin{equation}}
\def\ee{\end{equation}}
\def\bi{\begin{enumerate}}
\def\ei{\end{enumerate}}
\def\ba{\begin{array}}
\def\ea{\end{array}}
\def\bea{\begin{eqnarray}}
\def\eea{\end{eqnarray}}
\def\ben{\begin{enumerate}}
\def\een{\end{enumerate}}
\def\linebr{\linebreak}

\def\sl{\it}

\title{Cohomogeneity one disk bundles with normal homogeneous collars}

\author{Lorenz J. Schwachh\"ofer\thanks{Supported by the Schwerpunktprogramm Differentialgeometrie of the Deutsche
Forschungsgesellschaft} , Kristopher Tapp}




\date{ }
\maketitle


\begin{abstract}
We consider cohomogeneity one homogeneous disk bundles and adress the question when these admit a nonnegatively curved\footnote{Throughout this article, the term ``curvature'' refers to the sectional curvature.} invariant metric with normal collar, i.e., such that near the boundary the metric is the product of an interval and a normal homogeneous space. If such a bundle is not (the quotient of) a trivial bundle, then we show that its rank has to be in $\{2,3,4,6,8\}$. Moreover,  we give a complete classification of such bundles of rank $6$ and $8$, and a partial classification for rank $3$.
\end{abstract}

\section{Introduction}

The search for manifolds of nonnegative curvature is one of the classical problems in Riemannian geometry. One source of examples has been compact Lie groups and their quotients, including homogeneous spaces and biquotients. In addition to these, one has examples formed by glueing two manifolds along a common boundary, first done by J.~Cheeger~(\cite{Ch}). For a detailed survey on the known techniques and examples, we refer to \cite{Z}.

A large family of nonnegatively curved metrics was recently obtained by K.~Grove and W.~Ziller, who investigated closed cohomogeneity one manifolds with two singular orbits~(\cite{GZ1}). Each such manifold is obtained by the glueing along a principal orbit of two {\em cohomogeneity one homogeneous vector bundles}, i.e., bundles of the form
\[
M := G \times_K V \longrightarrow G/K,
\]
where $K \subset G$ are compact Lie groups, and $K$ acts transitively on the unit sphere of a finite dimensional Euclidean vector space $V$ by some orthogonal representation, with isotropy group henceforth denoted $H\subset K$. Grove and Ziller showed that any cohomogeneity one homogeneous vector bundle of rank at most two (that is, $\dim_\R V \leq 2$) admits an invariant nonnegatively-curved metric with a {\em normal homogeneous collar}, i.e., outside a compact set, the metric is $G$-equivariantly isometric to the Riemannian product of an interval and $G/H$ with a normal homogeneous metric. 
Thus, any cohomogeneity one manifold whose singular orbits are of codimension at most two admits a $G$-invariant metric of nonnegative curvature. 

In this article, we address the question of which cohomogeneity one homogeneous vector bundles of rank higher than two admit invariant metrics with nonnegative curvature and normal homogeneous collar. This question is of interest because, for a compact cohomogeneity one manifold $N$ with group diagram $H\subset\{K_+,K_-\}\subset G$, if the vector bundles associated to both halves of this diagram admit such metrics, then $N$ admits an invariant nonnegatively curved metric with  normal homogeneous principle orbits in the middle. Not all bundles admit such metrics because there are cohomogeneity one manifolds which are known not to admit any invariant nonnegatively-curved metric~(\cite{GVWZ}). 

One class of bundles which {\em does }Êadmit nonnegatively curved $G$-invariant metrics with normal homogeneous collar are what we call {\em essentially trivial bundles} by which we mean bundles of the form
\[
M := (G \times L) \times_{K \times L} V \longrightarrow G/K,
\]
where the action of $\{1\} \times L$ on $V$ is transitive on the unit sphere.  Thus, essentially trivial bundles are quotients of the trivial bundle $G \times V = (G \times L) \times_L V$ under the action of a subgroup $K \subset G$, so that the existence of such a metric follows from a constuction in \cite{ST}; cf. Corollary~\ref{cor:almost-trivial} for details. For example, all cohomogeneity one homogeneous vector bundles of rank one are essentially trivial.

Apart from these, our results show that bundles admitting such metrics are scarce.

\begin{thm} \label{thm:rank-restriction}
Let $M := G \times_K V \ra G/K$ be a cohomogeneity one homogeneous vector bundle which is essentially non-trivial. If $M$ admits a $G$-invariant metric with nonnegative curvature and normal homogeneous collar, then the rank of this bundle must be in $\{2, 3, 4, 6, 8\}$.
\end{thm}

As was previously mentioned, all rank two bundles admit a $G$-invariant metric of nonnegative curvature and normal homogeneous collar by \cite{GZ1}. In the higher rank case, the situation is much more restricted. For rank eight, we have the following complete classification.

\begin{thm} \label{thm:rank-8}
Let $M := G \times_K V \ra G/K$ be a $G$-irreducible cohomogeneity one homogeneous vector bundle which is essentially non-trivial and such that $\dim_\R V = 8$. Then $M$ admits a $G$-invariant metric of nonnegative curvature with normal homogeneous collar if and only if $M$ is finitely $G$-equivariantly covered by one of the following:

\bi
\item 
$Spin(p+9) \times_{Spin(8)} \R^8$ for $p \in \{0,1,2\}$, where $Spin(8)$ acts on $\R^8$ by a spin representation, and $Spin(8) \subset Spin(p+9)$ is the lift of the standard inclusion $SO(8) \subset SO(p+9)$.
\item 
$Spin(p+8) \times_{Spin(7)} \R^8$ for $p \in \{0,1\}$, where $Spin(7)$ acts on $\R^8$ by the spin representation, and $Spin(7) \subset Spin(p+8)$ is the lift of the standard inclusion $SO(7) \subset SO(p+8)$.
\item
$Spin(7) \times_{Spin(6)} \C^4$, with the standard representation of $Spin(6) \cong SU(4)$ on $\C^4$.
\item
A quotient of one of the preceding examples:
\bi
\item
$(Spin(p+9) \cdot G') \times_{Spin(8) \cdot H'} \R^8$ for $p \in \{1,2\}$ and an arbitrary compact Lie group $G'$ and $H' \subset Spin(p+1) \cdot G'$ with $H' \not\subset G'$, which acts trivially on $\R^8$.
\item
$(Spin(9) \cdot G') \times_{Spin(7) \cdot  H'} \R^8$ for an arbitrary compact Lie group $G'$ and $H' \subset Spin(2) \cdot G' = S^1 \cdot G'$ with $H' \not\subset G'$, which acts trivially on $\R^8$.
\item
$(Spin(7) \cdot G') \times_{Spin(6) \cdot S^1 \cdot H'} \C^4$ for an arbitrary compact Lie group $G' \supset S^1 \cdot H'$, where $S^1 \subset G'$ acts on $\C^4$ by multiples of the identity, and $H' \subset G'$ acts trivially.
\ei
\ei
\end{thm}

Here, for Lie groups $L_1, L_2$, we denote by $L_1 \cdot L_2$ the quotient of $L_1 \times L_2$ by a finite subgroup of the center. The term {\em $G$-irreducible} means that $M$ is not $G$-equivariantly finitely covered by a bundle of the form $(G_1/H_1) \times M'$ with $\dim(G_1/H_1) > 0$ and $M'$ a cohomogeneity one homogeneous vector bundle.  This hypothesis is natural because, for $G$-reducible bundles, our problem easily reduces to deciding whether $M'$ admits such a metric.

Glueing together two disk bundles of the second type of Theorem \ref{thm:rank-8} with $p = 0$, after applying outer automorphisms of $Spin(8)$, we conclude that the primitive cohomogeneity one manifold given by the group diagram $G_2 \subset \{Spin_+(7), Spin_-(7)\} \subset Spin(8)$ admits a metric of nonnegative curvature with a totally geodesic normal homogeneous principal orbit. However, this manifold is diffeomorphic to the sphere $S^{15}$ (\cite{GWZ}).

For rank $6$ bundles, we obtain the following

\begin{thm} \label{thm:rank-6}
Let $M := G \times_K V \ra G/K$ be a $G$-irreducible cohomogeneity one homogeneous vector bundle which is essentially non-trivial and such that $\dim_\R V = 6$. Then $M$ admits a $G$-invariant metric of nonnegative curvature with normal homogeneous collar if and only if $M$ is finitely $G$-equivariantly covered by one of the following:

\bi
\item
$SU(5) \times_{SU(4)} \R^6$, with the irreducible action of $SU(4) \cong Spin(6)$ on $\R^6$.
\item
$(SU(5) \cdot G') \times_{SU(4) \cdot H'} \R^6$ for an arbitrary compact Lie group $G'$ and $H' \subset S^1 \cdot G'$ with $S^1 \subset SU(5)$ being the centralizer of $SU(4)$, and $H' \not\subset G'$ acts trivially on $\R^6$.
\ei
\end{thm}

To describe our results for rank three and four bundles, we require some notation for subgroups of the exceptional Lie group $G_2$.  Let $SO(4) \subset G_2$ and $SU(3) \subset G_2$ be the isotropy groups of the symmetric space $G_2/SO(4)$ and the sphere $S^6 = G_2/SU(3)$, respectively. After conjugating these groups appropriately, their intersection can be made isomorphic to $U(2)$, and we let $SU(2)_1 \subset SO(4) \cap SU(3) \subset G_2$ be the simple part of this intersection. Note that $SU(2)_1 \subset SO(4)$ is normal, and we denote its centralizer in $G_2$ by $SU(2)_3 \subset SO(4) \subset G_2$. (The subscripts of the $SU(2)$-subgrous of $SO(4)$ denote their maximal weight for the isotropy representation of $G_2/SO(4)$.) Using this notation, we can make the following statement about the rank three case.

\begin{thm} \label{thm:rank-3}
Let $M := G \times_K V$ be a $G$-irreducible cohomogeneity one homogeneous vector bundle which is essentially non-trivial such that $\dim_\R V = 3$. If $M$ admits a nonnegatively-curved $G$-invariant metric with normal homogeneous collar, then $M$ must be finitely $G$-equivariantly covered by one of the following.
\bi
\item
$M_1 = G_2 \times_{SO(4)} \su(2)_3$ , where $SO(4)$ acts on $\su(2)_3 \lhd \so(4)$ by the adjoint representation.
\item
$M_2 = (Sp(p+1) \cdot G') \times_{Sp(1) \cdot H'} \sp(1)$ with $H' \subset Sp(p) \cdot G'$, where $Sp(1) \cdot H'$ acts on $\sp(1) \lhd \sp(1) \oplus \h'$ by the adjoint representation.
\ei
Further, $M_1$ admits such a metric, as does $M_2$ with $G' = 1$ and $H' = Sp(p)$.
\end{thm}

Finally, in the rank four case,  we have the following examples,  which are all related to those in Theorem~\ref{thm:rank-3}.

\begin{thm} \label{thm:rank-4}
The following cohomogeneity one homogeneous vector bundles (orbifold bundles, respectively) of rank four admit $G$-invariant metrics of nonnegative curvature and normal homogeneous collar:
\bi
\item
$G_2 \times_{SO(4)} (\H/\pm 1)$, where $SU(2)_1 \subset SO(4)$ acts trivially and $SU(2)_3 \subset SO(4)$ by left multiplication on $\H/\pm 1$. Note that this is an orbifold bundle only.
\item
$(G_2 \times G')\times_{SO(4) \times SU(2)'} (\H/\pm 1)$, where $SU(2)_1 \subset SO(4)$ acts trivially and $SU(2)_3 \subset SO(4)$ by left multiplication, whereas $SU(2)' \subset G'$ acts by right multiplication on $\H/\pm 1$ with $G'$ arbitrary. Note that these are orbifold bundles only.
\item
$Sp(p+1) \times_{Sp(1) \times Sp(p)} \H$ where $Sp(p)$ acts trivially and $Sp(1)$ by left multiplication on $\H$.
\item
$(Sp(p+1) \times G') \times_{Sp(1) \times Sp(p) \times Sp(1)'} \H$ where $Sp(p)$ acts trivially and $Sp(1)$ by left multiplication, whereas $Sp(1)' \subset G'$ acts by right multiplication on $\H$ with $G'$ arbitrary.
\ei
\end{thm}


It is a pleasure to thank Karsten Grove and Wolfgang Ziller for many inspiring and clarifying discussions on this work. We are also grateful to the American Institute of Mathematics for hospitality and funding at a workshop on nonnegative curvature in September 2007, where portions of this work were discussed.

\section{Construction of metrics}\label{S:constructions}

In this section, we construct invariant nonnegatively curved metrics with normal homogeneous collars on all of the bundles which are claimed to admit such metrics in Theorems~\ref{thm:rank-8},\ref{thm:rank-6},\ref{thm:rank-3},\ref{thm:rank-4}.

We adopt the following notation for the remainder of the article.  Let $M := G \times_K V\ra G/K$ be a cohomogeneity one homogeneous vector bundle, and let $H \subset K$ denote the isotropy group of the transitive action of $K$ on the unit sphere in $V$.  Let $\h \subset \k \subset \g$ denote the Lie algebras of $H\subset K\subset G$, and let $Q$ be an Ad-invariant inner product on $\g$. We denote the $Q$-orthogonal decompositions by
\be \label{eq:decompose}
\k = \h \oplus \m,\,\,\,\,\,\, \g = \k \oplus \s,\, \mbox{ and we let } \p := \m \oplus \s.
\ee
All of our examples come from:
\begin{thm} \label{thm:sufficient}
If there exists $C > 0$ such that for all $X = X_\m + X_\s, Y = Y_\m + Y_\s \in \p$  we have the inequality
\be \label{eq:sufficient}
|X_\m \w Y_\m| \leq C |[X, Y]|,
\ee
then $M$ admits a nonnegatively curved $G$-invariant metric with  normal homogeneous collar.
\end{thm}

\begin{proof} In \cite[Theorem 5.1]{ST}, it was shown that for each linear action of $K$ on $V$ which is transitive on the unit sphere $S^l\subset V$, there exists a nonnegatively curved $K$-invariant metric on $V$ which outside a compact set has the form $g_V = dt^2 + f(t)^2 g_Q^{S^l}$ where $g_Q^{S^l}$ is the normal homogeneous metric on $S^l$ induced by the bi-invariant metric $Q$ on $K$, and $f$ is a smooth function with $f' > 0$. From the curvature formula for warped products it follows that $f'' \leq 0$, and that this metric remains nonnegatively-curved if we replace $f$ by any other function with $f'' \leq 0$. Therefore, we can achieve that $f$ is constant for large $t$. Thus, we may assume that $g_V$ has the form
\[
g_V = dt^2 + c_0^2 g_Q^{S^l}
\]
outside $B_R(0) \subset V$ for some $R > 0$, where $c_0 > 0$ can be chosen to be arbitrarily large.

Next, let $g_\eps$ be the left-invariant metric on $G$ given by
\be \label{eq:geps}
g_\eps = Q|_\s + (1 + \eps) Q|_\k.
\ee
Then $K$ acts by isometries on $(G \times V, g_\eps + g_V)$ as $k\star(g,v):=(gk^{-1},kv)$. There is a unique induced metric on $M$ such that the canonical submersion $G \times V \ra M$ is Riemannian.  This metric on $M$ is invariant under the $G$-action induced from the left action on the first factor of $G\times V$.

Furthermore, it was shown in \cite[Theorem 0.2]{ST1} that under our hypothesis, the tangent planes in $\p = \s \oplus \m \subset \g$ have nonnegative curvature for $(G, g_\eps)$ for sufficiently small $\eps > 0$. Thus, the horizontal planes of the submersion $(G \times V, g_\eps + g_V) \ra M$ all have nonnegative curvature, hence so does the induced metric on $M$ by O'Neill's formula.

If we choose $c_0^2 := (1 + \eps)/\eps$, then it follows from the ``scale-up/scale-down'' metric construction of~\cite[Lemma~2.1]{GZ1} that $M$ has a normal homogeneous collar.
\end{proof}

\begin{cor} \label{cor:almost-trivial} Let $M = G \times_K V \ra G/K$ be a cohomogeneity one homogeneous disc bundle for which (\ref{eq:sufficient}) holds, and let $G'$ be an arbitrary compact Lie group. Let $K' \subset Norm_G(K) \times G'$ be a closed subgroup containing $K \cong K \times \{1\}$, and let $K'$ act orthogonally on $V$ extending the action of $K \subset K'$. Then the bundle
\[
M' := (G \times G') \times_{K'} V = (M \times G')/(K'/K)
\]
admits an invariant metric of nonnegative curvature and normal homogeneous collar. In particular, this shows that every  {\em essentially trivial} homogeneous disk bundles admits such a metric.
\end{cor}

\begin{proof}
We assert that the $K$-invariant metric on $V$ from \cite[Theorem~5.1]{ST} used in the proof of Theorem~\ref{thm:sufficient} is invariant not only under $K$ but under $Norm_{O(V)}K$. This is due to the fact that for all but one type of transitive actions on spheres (\cite{Montgomery}; cf Table~1), the isotropy representation splits into mutually inequivalent irreducible representations, so that any $K$-invariant metric on $V$ is $Norm_{O(V)}(K)$-invariant by Schur's Lemma. The only exception is $S^{4p+3} = Sp(p+1)/Sp(p)$ in which case the isotropy has a $3$-dimensional trivial summand, and $Norm_{O(4p+4)}Sp(p) = Sp(1) \cdot Sp(p)$. However, looking at the construction in \cite{ST}, it follows that this metric is $Norm_{O(4p+4)}Sp(p)$-invariant as well which shows the assertion.

Consider $(G \times G' \times V, g_\eps + g' + g_V)$ with $g_\eps$ and $g_V$ defined in the proof of Theorem~\ref{thm:sufficient} and $g'$ on $G'$ biinvariant. Since $K \subset K'$ is a {\em normal} subgroup, our hypothesis implies that $K'$ acts on $V$ by elements of $Norm_{O(V)}K$ and hence by isometries of $g_V$. Also, $g_\eps + g'$ is invariant under right multiplication  of $G \times G'$ by $K' \subset Norm_G(K) \times G'$. Since $M' = (G \times G') \times_{K'} V$ is the corresponding quotient, and (\ref{eq:sufficient}) guarantees that the horizontal planes of the submersion $(G \times G' \times V, g_\eps + g' + g_V) \ra M'$ have nonnegative curvature, it follows that the induced metric on $M'$ has the asserted properties.

In order to show that this applies to essentially trivial bundles , it remains to verify that $G \times V = (G \times L) \times_L V$ satisfies the hypothesis (\ref{eq:sufficient}) if $L \subset O(V)$ acts transitively on the unit sphere $S^l \subset V$. Indeed, the normal homogeneous metric on $S^l = L/H$ has positive curvature by \cite{Ber}. Thus,
\[
|X_\m \w Y_\m| \leq C |[X_\m, Y_\m]| = C |[X, Y]_\l| \leq C |[X, Y]|,
\]
where $\frac14 C^{-1} > 0$ is a lower curvature bound for $S^l$.
\end{proof}

It remains to prove, for each of the bundles which are claimed to admit such metrics in Theorems~\ref{thm:rank-8},\ref{thm:rank-6},\ref{thm:rank-3},\ref{thm:rank-4}, that either hypothesis (\ref{eq:sufficient}) is satisfied or Corollary~\ref{cor:almost-trivial} applies. The first three examples of Theorem~\ref{thm:rank-8} and the first example of Theorem~\ref{thm:rank-6} correspond to the following triples $H \subset K\subset G$:
\begin{enumerate}
\item $Spin_\pm(7) \subset Spin(8) \subset Spin(p+9)$ for $p \in \{0, 1, 2\}$, where the first inclusion is by the spinor representation (cf. section~\ref {sec:octonions} for details), and the second is the lift of the standard embedding $SO(8) \subset SO(p+9)$,
\item $G_2\subset Spin(7) \subset Spin(p+8)$ for $p \in \{0, 1\}$, where the second inclusion is the lift of the standard embedding $SO(7) \subset SO(p+8)$,
\item $SU(3)\subset SU(4)\cong Spin(6)\subset Spin(7)$,
\item $Spin(5) \cong Sp(2) \subset Spin(6) \cong SU(4) \subset SU(5)$.
\end{enumerate}
That these triples satisfy hypothesis (\ref{eq:sufficient}) was proven in~\cite{ST1}.  In all cases, this was achieved by verifying that
\[
[\s,\s]\cap[\m,\m]=\{0\},
\]
which implies the hypothesis.

The remaining examples of Theorems~\ref{thm:rank-8} and~\ref{thm:rank-6} now all are obtained by applying Corollary~\ref{cor:almost-trivial} to the above examples.

Likewise, the second and fourth examples in Theorem~\ref{thm:rank-4} follow by applying Corollary~\ref{cor:almost-trivial} to the first and third examples, respectively. The latter correspond to the triples:
\begin{enumerate}
\item $SU(2)_1\subset SO(4)\subset G_2$,
\item $Sp(p) \subset Sp(1)\cdot Sp(p) \subset Sp(p+1)$.
\end{enumerate}
The first triple was verified in~\cite{ST1}.  The second triple satisfies the hypothesis because $Sp(p+1)/Sp(p) = S^{4p+3}$ is a sphere whose normal homogeneous metric has positive curvature with lower bound, say, $4 \eps$. Then it follows that $|[X, Y]| \geq \eps |X \w Y|$ for all $X, Y \in \m \oplus \s$, and since $|X \w Y| \geq |X_\m \w Y_\m|$, the hypothesis of Theorem~\ref{thm:sufficient} is satisfied for $C := \eps^{-1}$.

The examples  in Theorem~\ref{thm:rank-3} correspond to the triples:
\begin{enumerate}
\item $SU(2)_1 \cdot S^1 \subset SO(4) \subset G_2$,
\item $S^1 \cdot Sp(p) \subset Sp(1)\cdot Sp(p) \subset Sp(p+1)$,
\end{enumerate}
each of which is formed from a rank 4 example by enlarging $H$; this change obviously maintains condition (\ref{eq:sufficient}).

\section{Necessary conditions for normal homogeneous collars} \label{sec:curvature}

For the remainder of this paper, we assume that $M= G \times_K V$ is endowed with a $G$-invariant nonnegatively curved metric such that there exists a principal $G$-orbit in $M$ which is totally geodesic and normal homogeneous, i.e., it is induced by an Ad-invariant inner product $Q$ on $\g$. This is slightly weaker than assuming that $M$ has a normal homogeneous collar, but it will imply the same rigidity.

As in (\ref{eq:decompose}), we have the $Q$-orthogonal decomposition $\g = \h\oplus\p = \h \oplus \m \oplus \s$. The goal of this section is to prove:

\begin{thm} \label{thm:necessary}
If $\m_1 \subset \m$ is any non-trivial $Ad_H$-irreducible subspace such that $\m$ contains no irreducible factor equivalent to $\m_1$, then there exists a constant $C > 0$ such that for all $X = X_\m + X_\s, Y = Y_\m + Y_\s \in \m_1 \oplus \s$ we have:
\be \label{eq:condition}
|X_\m \w Y_\m| \leq C |[X, Y]|.
\ee
\end{thm}

When $K/H$ is isotropy irreducible, the choice $\m_1=\m$ yields a converse to Theorem~\ref{thm:sufficient}. Evidently, (\ref{eq:condition}) implies that for $X, Y \in \m_1 \oplus \s$, we can have $[X, Y] = 0$ only if $X_\m, Y_\m$ are linearly dependent. However, the converse implication is false in general; cf. Remark~\ref{rem:limit}.

Towards proving this theorem, first notice that the disk bundle $D \subset M$ which is bounded by the totally geodesic principal orbit is totally convex and thus has the singular orbit $\Sigma\subset D$ as its soul, since $\Sigma$ is a closed submanifold equidistant from the boundary.  We can choose a point $p\in\Sigma$ and a unit-speed normal geodesic $c:[0,l]\ra M$ with $c(0)=p$ and $c(l)\in\partial D$ such that $K$ is the stabilizer of $p$ and $H$ is the stabilizer of $c(t)$ for all $t\in(0,l]$.  For each $t\in[0,l]$, there is a self-adjoint map $\varphi_t:\p\ra\p$ such that
$$\lb X^*,Y^*\rb_{c(t)} = Q(X,\varphi_t Y)$$
for all $X,Y\in\p$, where $X^*,Y^*$ denotes the action fields of $X,Y$.  Each $\varphi_t$ is positive definite, with the exception that $\varphi_0|_\m = 0$.  Thus, $D\backslash\Sigma$ is $G$-equivariantly isometric to the warped product $((0,l]\times(G/H),dt^2 + g_{\varphi_t})$, where $\{g_{\varphi_t}\}$ is the family of homogeneous metrics on $G/H$ determined by $\{\varphi_t\}$.  By assumption, $\varphi_l=Id$.

\begin{lem} \label{lem:soul}
For each $t\in[0,l]$, $\varphi_t|_{\s} = Id$.  In particular, the soul $\Sigma$ is normal homogeneous.
\end{lem}
\begin{proof}
For each $X\in\p$, the action field $X^*$ on $M$ is Killing, so its restriction to the geodesic $c$ is a Jacobi field, which we denote as $X_t$.  Thus, $X\mapsto X_t$ is an identification of $\p$ with a family of Jacobi fields along $c$.  This family has the property that
$$\lb X_t,Y'_t\rb = \lb X'_t,Y_t\rb \textrm{ for all }X,Y\in\p,$$ because these derivatives are determined by the second fundamental forms of the principle orbits.  For any $X\in\p$ such that $X_t$ is parallel, we must have that $X\in\s$.  This is because, for all $Y\in\m$, we know that $Y_0=0$ and
$$\frac{d}{dt}\lb X_t,Y_t\rb = \lb X_t',Y_t\rb + \lb X_t,Y_t'\rb = 2\lb X'_t,Y_t\rb = 0,$$
thus, $0=\lb X_l,Y_l\rb = Q(X,\varphi_l Y) = Q(X,Y)$.

Conversely, we wish to show that each element of $\s$ determines a parallel Jacobi field along $c$, which will complete the proof.  By Perelman's Theorem~\cite{Per}, each $\overline{X}_0\in T_p\Sigma$ extends to a parallel Jacobi field, $\overline{X}_t$, along $c$, with $\overline{X}_l$ tangent to $\partial D$.  There exists some $X\in\p$ with $X_l=\overline{X}_l$.  Since $\partial D$ is totally geodesic, we also have $X'(l) = 0 = \overline{X}'_l$, so the two Jacobi fiels must agree: $X_t=\overline{X}_t$ for all $t$.  Since $X_t$ is parallel, we know from above that $X\in\s$, and in fact $X$ is the unique vector in $\s$ identified with $\overline{X}_0$ via the identification $\s\cong T_p\Sigma$ given by action fields.  Thus, for all $X\in\s$, $X_t$ is a parallel Jacobi field.
\end{proof}

Before we prove Theorem~\ref{thm:necessary}, we need the following

\begin{lem} \label{lem:technical}
Let $H \subset K \subset G$ be as above, and let $\m_1 \subset \m$ be a linear subspace. Then there are polynomials $\la_i$ for $i = 1,2,3$ with the following property. If $\phi: \p \ra \p$ is a positive definite $Ad_H$-equivariant map with $\phi|_\s = Id$ and $\phi|_{\m_1} = (1 - h)^{-1} Id$ for some $h \in (-\infty, 1)$, and if we let $\psi := Id - \phi^{-1}$, then for all $X = X_\m + X_\s, Y = Y_\m + Y_\s \in \m_1 \oplus \s$ we have
\[
k^{\phi}(X', Y') \leq \la_1(|\psi|) |[X, Y]|^2 + \la_2(|\psi|) |[X, Y]| \cdot |X_\m \w Y_\m| + \la_3(|\psi|) h^2 |X_\m \w Y_\m|^2,
\]
where $X' := \phi^{-1} X$, $Y' := \phi^{-1} Y$ and $k^{\phi}$ denotes the unnormalized curvature of $(G/H, g_\phi)$.
\end{lem}

\begin{proof}
For the proof, we use the description of the unnormalized curvature established in \cite[Proposition~2.4]{ST1}. There, $k(t)$ denotes the curvature of the metric $\phi_t = (Id - t \psi)^{-1}$. We need this formula for $t = 1$ which reads
\be \label{eq:unnormalized}
k^{\phi}(X', Y') = \al + \beta + \gamma + \delta - \frac 34 |D^\p|^2_{g_1} \leq \al + \beta + \gamma + \delta,
\ee
where
\begin{eqnarray*}
\alpha & = & |[X,Y]^\h|^2 + \frac 14|[X,Y]^\p|^2\\
\beta & = & -\frac 34\lb \psi[X,Y],[X,Y]\rb - \frac 32\lb[X,Y]^\h,A\rb\\
\gamma & = & -\frac 34|\psi[X,Y]|^2 + \frac 32\lb\psi[X,Y],A\rb -\frac 32\lb[X,Y]^\m,B\rb +\frac 34|A^\h|^2 \\
\delta & = &  -\frac 34\lb \psi^3 [X,Y],[X,Y]\rb+\frac 32 \lb \psi^2 [X,Y],A \rb
              -\frac 32 \lb \psi [X,Y],B\rb \\
       &&     -\frac 34 \lb \psi A,A\rb-\frac 14 \lb \psi C, C\rb+\lb \psi[\psi X,X], [\psi Y,Y]\rb+\lb A,B\rb -\frac 32\lb A^\h,B\rb
\end{eqnarray*}
and
\begin{eqnarray*}
A & = & [\psi X,Y]+[X,\psi Y] = h (2 [X_\m, Y_\m] + [X_\m, Y_\s] + [X_\s, Y_\m]),\\
B & = & [\psi X,\psi Y] = h^2 [X_\m, Y_\m] \in \k,\\
C & = & [\psi X,Y]-[X,\psi Y] = h ([X_\m, Y_\s] - [X_\s, Y_\m]) \in \s.
\end{eqnarray*}
Let us label the two norms
\[
N_1 := |[X, Y]| \mbox{ and } N_2 := |X_\m \w Y_\m|.
\]
There is a constant $\la > 0$ such that for all $X_\m, Y_\m \in \m_1$ we have
\[
|[X_\m, Y_\m]| \leq \la N_2
\]
where $\la$ is the norm of the linear map $[\ ,\ ]: \Lambda^2 \m_1 \ra \k$. Thus,
\[ \ba{rll}
|A^\h| \leq |A^\k| & = & 2 |h|\ |[X_\m, Y_\m]| \leq 2 \la |h| N_2,\\
|B| & \leq & \la h^2 N_2.
\ea \]
Moreover, since $C \in \s$ we have $\psi C = 0$, and $[\psi X, X] = h [X_\m, X_\s] \in \s$ so that $\psi[\psi X, X] = 0$. Thus, using in addition that $|h| \leq |\psi|$, we obtain the following estimates.
\begin{eqnarray*}
\alpha & \leq & N_1^2\\ \\
\beta & \leq & \frac34 |\psi| N_1^2 + \frac32 N_1 |A^\h| \leq \frac34 |\psi| N_1^2 + 3 \la |h| N_1 N_2
\leq 
\frac34 |\psi| N_1^2 + 3 \la |\psi| N_1 N_2 \\ \\
\gamma & \leq & \frac 32 N_1\ |\psi A^\k| + \frac 32 N_1\ |B| +\frac 34|A^\h|^2\\
& \leq & 3 \la |h| |\psi| N_1 N_2 + \frac 32 \la h^2 N_1 N_2 + 3 \la^2 h^2 N_2^2 \\
& \leq & \frac 92 \la |\psi|^2 N_1 N_2 + 3 \la^2 h^2  N_2^2\\ \\
\delta & \leq & \frac 34 |\psi|^3 N_1^2 + \frac 32 N_1\ |\psi^2 A^\k| + \frac 32 |\psi| N_1\ |B| + \frac 34 |\psi A^\k|\ |A^\k| + \frac 52 |A^\k|\ |B|\\
& \leq & \frac 34 |\psi|^3 N_1^2 + (3 \la |h| |\psi|^2 + \frac 32 \la h^2 |\psi|) N_1 N_2\\ &&
 + (3 \la^2 h^2 |\psi| + 5 \la^2 |h|^3) N_2^2\\
& \leq & \frac 34 |\psi|^3 N_1^2 + \frac 92 \la |\psi|^3 N_1 N_2 + 8 \la^2 h^2 |\psi| N_2^2
\end{eqnarray*}
In the estimate of $\gamma$, we dropped the first term as it is nonpositive. Substituting all of this into (\ref{eq:unnormalized}) we obtain
\begin{eqnarray*}
k^{\phi}(X', Y') & \leq & (1 + \frac34 |\psi| + \frac 34 |\psi|^3) N_1^2 + (3 \la |\psi| + \frac 92 \la |\psi|^2 + \frac 92 \la |\psi|^3) N_1 N_2\\ & & + (3 \la^2 + 8 \la^2 |\psi|) h^2 N_2^2,
\end{eqnarray*}
which shows the claim.
\end{proof}

We are now ready to prove Theorem~\ref{thm:necessary}.

\begin{proof}
Let $\m_1\subset\m$ be as in the theorem. By Schur's Lemma, $\phi_t|_{\m_1}= f(t) \cdot Id$ for some smooth function $f:(0,l]\ra\R^+$, and we let $\psi_t := Id - \phi_t^{-1}$. Thus, by Lemma~\ref{lem:soul}, we have $\phi_t|_\s = Id$ and $\phi_t|_{\m_1} = (1 - h(t))^{-1} Id$ for $h(t) := 1 - f(t)^{-1}$. Moreover, $\psi_l = 0$ since $t = l$ corresponds to the normal homogeneous principal orbit. Identifying $M\backslash\Sigma\cong((0,l]\times(G/H),dt^2+ g_{\phi_t})$, the unnormalized curvature $k^M$ of the tangent plane in $M$ at $c(t)$ spanned by the action fields of $X'$ and $Y'$ equals:
\[
k^M(X',Y') =  k^{\phi_t}(X', Y') + II(X',Y')^2 - II(X',X')\ II(Y', Y').
\]
The second fundamental form of this warped product metric satisfies:
\[
II(X', Y') = \frac12 Q(X', \dot \phi_t Y') = \frac12 h'(t) Q(X_\m, Y_\m).
\]
Thus, by Lemma~\ref{lem:technical},
\begin{eqnarray*}
k^M(X',Y') & =  & k^{\phi_t}(X', Y') -  \frac14 h'(t)^2 |X_\m \w Y_\m|^2\\Ê\\
& \leq & \la_1(|\psi_t|) |[X, Y]|^2 + \la_2(|\psi_t|) |[X, Y]| \cdot |X_\m \w Y_\m|\\ && + \la_3(|\psi_t|) h(t)^2 |X_\m \w Y_\m|^2 -  \frac14 h'(t)^2 |X_\m \w Y_\m|^2.
\end{eqnarray*}
Since we assume that the metric on $M$ has nonnegative curvature, this implies for all $X = X_\m + X_\s$, $Y = Y_\m + Y_\s \in \m_1 \oplus \s$ with $X_\m \w Y_\m \neq 0$
\be \label{eq:estimate1}
\ba{ll}
0 \leq \frac{k^M(X',Y')}{|X_\m \w Y_\m|^2} \leq & \la_1(|\psi_t|) \rho(X, Y)^2 + \la_2(|\psi_t|) \rho(X, Y)\\ & + \la_3(|\psi_t|) h(t)^2 -  \frac14 h'(t)^2,
\ea
\ee
where $\rho(X, Y) := \frac{|[X, Y]|}{|X_\m \w Y_\m|}$. Suppose the theorem is false, which means that
\be \label{eq:infimum}
\inf\left\{ \rho(X, Y) \mid X = X_\m + X_\s, Y = Y_\m + Y_\s \in \m_1 \oplus \s, X_\m \w Y_\m \neq 0 \right\} = 0.
\ee
Since (\ref{eq:estimate1}) must hold for {\em all} $X, Y \in \m_1 \oplus \s$  with $X_\m \w Y_\m \neq 0$, (\ref{eq:infimum}) implies
\be \label{eq:estimate2}
0 \leq  \la_3(|\psi_t|) h(t)^2 -  \frac14 h'(t)^2.
\ee
 Suppose that there is a $t_0 \in (0, l]$ such that $h(t_0) \neq 0$, and let $t_1 := \min\{t \in (t_0, l] \mid h(t) = 0\}$ which exists as $h(l) = 0$. Since $\la_3$ is continuous, it follows that $\la_3(|\psi_t|) \leq C$ for all $t \in [t_0, t_1]$ and some constant $C > 0$, and hence by (\ref{eq:estimate2}),
\[
\left( \log(h(t))' \right)^2 = \frac{h'(t)^2}{h(t)^2} \leq 4 \la_3(|\psi_t|) \leq 4 C
\]
 for all $t \in [t_0, t_1)$, i.e., $\log(h(t))$ has bounded derivative for all such $t$. On the other hand, $\lim_{t \nearrow t_1} \log h(t) = -\infty$ which is impossible.

This shows that we must have $h(t) \equiv 0$ and hence $f(t) \equiv 1$ on $(0, l]$. Note that for all $X \in \m$ we have $\lim_{t \ra 0} (X_\m)_{c(t)} = 0$. Thus, for $X_\m \in \m_1$ we have $0 = \lim_{t \ra 0} g_{c(t)}(X_\m, X_\m) = \lim_{t \ra 0} Q(X_\m, \phi_t X_\m) = |X_\m|_Q^2$ which is a contradiction. 
\end{proof}

\section{Octonions and Triality} \label{sec:octonions}

In this section, we will collect some facts about the octonion numbers and triality which are well known (for a survey, see e.g. \cite{Baez}), and show that certain triples $H \subset K \subset G$ do not satisfy condition (\ref{eq:condition}) from Theorem~\ref{thm:necessary}. 

Let $\O \cong \R^8$ denote the octonian numbers, and let $G_2$ be the automorphism group of $\O$. Since $G_2$ stabilizes $1 \in \O$ and is orthogonal, it leaves $Im(\O) \cong \R^7$ invariant. We consider the $Ad_{G_2}$-invariant decomposition
\be \label{eq:spin(8)}
\so(8) = \so(\O) = \g_2 \oplus \{ L_q \mid q \in Im(\O) \} \oplus \{ R_q \mid q \in Im(\O) \} := \g_2 \oplus V_L \oplus V_R,
\ee
where $L_q, R_q: \O \ra \O$ denote multiplication from the left and the right, respectively. Indeed, $V_L, V_R$ both are $Ad_{G_2}$-invariant and not equal so that their intersection must vanish. Then a dimension count shows that the three summands on the right of (\ref{eq:spin(8)}) span all of $\so(8)$. Note, however, that $V_L$ and $V_R$ are not orthogonal; in fact, these spaces intersect at an angle of $\pi/3$.

Let $Spin_0(7) \subset Spin(8)$ be the subgroup obtained by the lift of the inclusion $SO(7) \subset SO(8)$ of endomorphisms stabilizing $1 \in \O$. Then all elements in its Lie algebra $\so_0(7) \subset \so(8)$ vanish on $1$, hence we obtain the orthogonal decomposition
\[
\so(8) = \underbrace{\g_2 \oplus \m_0}_{= \so_0(7)} \oplus \s_0, \mbox{ where } 
\ba{rl}
\m_0 & = \{ L_q - R_q \mid q \in Im(\O) \}, \mbox{ and }\\  \s_0 & = \{ L_q + R_q \mid q \in Im(\O) \}.
\ea
\]
As a $G_2$-module, the decomposition (\ref{eq:spin(8)}) can be written as
\[
\so(8) = \g_2 \oplus \left(\R^2 \ot (Im(\O))\right),
\]
hence any automorphism of $Spin(8)$ which leaves $G_2$ invariant determines an element of $O(2)$ acting on the $\R^2$-factor.

The {\em triality group} is defined as the group of outer automorphisms of $Spin(8)$. This group is isomorphic to the permutation group $S_3$. Moreover, each outer automorphism can be uniquely represented such that it acts on $G_2 \subset Spin(8)$ as the identity. Thus, by the preceding paragraph, the triality automorphisms induce a faithful homomorphism $S_3 \ra O(2)$ and hence, $S_3$ acts on $\R^2$ as the isometry group of an equilateral triangle. Therefore, the orbit of $Spin_0(7) \subset Spin(8)$ under the triality group consists of three subgroups
\[
Spin_0(7), Spin_\pm(7) \hookrightarrow Spin(8),
\]
and the Lie algebras $\so_\pm(7) \subset \so(8)$ of the latter induce the orthogonal decompositions
\[
\so(8) = \underbrace{\g_2 \oplus \m_\pm}_{= \so_\pm(7)} \oplus \s_\pm, \mbox{ where } 
\ba{rlllrl}
\m_+ & = \{ L_q + 2 R_q \mid q \in Im(\O) \}, & \mbox{ and } & \s_+ & = V_L,\\
\m_- & = \{ 2 L_q + R_q \mid q \in Im(\O) \}, & \mbox{ and } & \s_- & = V_R.
\ea
\]

Let $Spin(8) \subset Spin(9)$ be the lift of the inclusion $SO(8) \subset SO(9)$. Then $Spin_0(7) \subset Spin(8) \subset Spin(9)$ is the lift of the inclusion $SO(7) \subset SO(9)$, whereas $Spin_\pm(7) \subset Spin(8) \subset Spin(9)$ are conjugate subgroups which are the isotropy of $S^{15} = Spin(9)/Spin(7)$.

\begin{prop} \label{prop:octonions}
For the following triples $H \subset K \subset G$, condition (\ref{eq:condition}) of Theorem~\ref{thm:necessary} with $\m_1 = \m$ is violated.
\bi
\item
$G_2 \subset Spin_\pm(7) \subset Spin(p + 9)$ for $p \geq 0$,
\item
$G_2 \subset Spin_0(7) \subset Spin(p + 9)$ for $p \geq 1$,
\item
$G_2 \subset Spin_0(7) \subset F_4$,
\item
$Spin_\pm(7) \subset Spin(8) \subset Spin(p + 9)$ for $p \geq 3$,
\ei
where the second inclusions are given by the composition of the inclusion $Spin_i(7) \subset Spin(9)$ from above with the lift of the inclusion $SO(9) \subset SO(p + 9)$ in the first tow cases and with the isotropy of the Cayley plane $F_4/Spin(9)$ in the third case, and by the lift of $SO(8) \subset SO(p + 9)$ in the last case.
\end{prop}

\begin{proof}
Decompose $\R^{p + 9} = \O \oplus \R^{p+1}$. In the first two cases, we have the splitting
\be \label{eq:decompose-so(8+p)}
\so(p + 9) = \overbrace{\underbrace{\g_2}_{= \l} \oplus \underbrace{\m_i}_{=\m} \oplus \underbrace{\s_i}_{\subset \s}}^{= \so(8)} \oplus \underbrace{\so(p+1)}_{= \z(\l)} \oplus \underbrace{\O \ot \R^{p+1}}_{\subset \s},
\ee
where $Spin(8) \times Spin(p+1)$ acts on $\O \ot \R^{p+1}$ by the tensor representation of $SO(8) \times SO(p+1)$.

To show the assertion for the first triple, pick the orthonormal set $e_1 = 1, e_2 = i, e_3 = j \in \O$ and $e_4 \in \R^{p+1}$. Moreover, let $E_{rs} = -E_{sr} \in \so(8 + p)$ denote the rank two matrix with $E_{rs}e_r = e_s$ and $E_{rs}e_s = -e_r$, and define
\[
X := E_{12} + E_{34}, \mbox{ and } Y := E_{13} + E_{24}.
\]
Then $[X, Y] = 0$, and $E_{24}, E_{34} \in \O \ot \R^{p+1} \subset \s$. Moreover, $2 E_{12} = L_i + R_i$, $2 E_{13} = L_j + R_j \in \s_0$, so that $X, Y \in \g_2^\perp$. If the $\m_\pm$-components of $X$ and $Y$ were linearly dependent, then the span of $E_{12}, E_{13} \in \s_0$ would intersect $\s_\pm$ which is impossible. Thus, $X_{\m_\pm}, Y_{\m_\pm}$ are linearly independent which shows that (\ref{eq:condition}) does not hold.

For the second triple, we write $\O = \H \oplus \eps \H$ where $\H$ denote the quaternions. Pick an orthonormal basis $e_1, \ldots, e_4$ of $\eps \H$ by
\[
e_1 := \eps, e_2 := \eps k, e_3 := \eps i, e_4 := \eps j,
\]
and let $e_5, e_6 \in \R^{p+1}$ be othonormal. Let $L_i, R_j \in \so(7)$ be defined as above. Then $[L_i, R_j] \in \so(7)$ is the associator, whence $[L_i, R_j]|_\H = 0$ whereas $[L_i, R_j]|_{\eps \cdot \H} = 2 L_k|_{\eps \cdot \H}$. Thus, in the matrix notation from above, we may write w.r.t. the above basis
\[
[L_i, R_j] = - 2 (E_{12} + E_{34}).
\]
Moreover, $L_i e_2 = R_j e_1$ and $L_ i e_4 = R_j e_3$ so that $[L_i, E_{2r}] = [R_j, E_{1r}]$ and $[L_i, E_{4r}] = [R_j, E_{3r}]$ for $r = 5,6$. Thus, if we define
\[
X := L_i + \sqrt 2 (E_{15} + E_{36}) \mbox{ and } Y := R_j + \sqrt 2 (E_{25} + E_{46}) \in \m \oplus \s,
\]
then one verifies from here that $[X, Y] = 0$. But from (\ref{eq:decompose-so(8+p)}) it follows that $X_\m = (L_i)_\m = \frac12 (L_i - R_i)$ and $Y_\m = (R_j)_\m = -\frac12 (L_j - R_j)$, i.e., $X_\m, Y_\m$ are linearly independent  which shows that (\ref{eq:condition}) does not hold.

Now let us show that the third triple does not satisfy (\ref{eq:condition}). The Lie algebra ${\frak f}_4$ of $F_4$ can be decomposed as
\[
{\frak f}_4 = \so(8) \oplus V_8 \oplus \Delta_8^+ \oplus \Delta_8^-,
\]
where $V_8$ is the standard and $\Delta_8^\pm$ are the spin representations of $\so(8)$. Moreover, $\so(8) \oplus V_8 := \so_0(9) \subset {\frak f}_4$ and $\so(8) \oplus \Delta_8^\pm =: \so_\pm(9) \subset {\frak f}_4$ are Lie subalgebras corresponding to conjugate Lie subgroups $Spin_i(9) \subset F_4$ with $i \in \{0, \pm\}$ whose intersection is $Spin(8)$.

In particular, the inclusions $Spin_0(7) \subset Spin(8) \subset Spin_\pm(9)$ correspond to the isotropy inclusion of $S^{15} = Spin(9)/Spin(7)$ and hence, the triples $G_2 \subset Spin_0(7) \subset Spin_\pm(9)$ do {\em not} satisfy (\ref{eq:condition}) by the first case considered above, and hence, when enlarging $Spin_\pm(9)$ to $F_4$, (\ref{eq:condition}) remains violated.

Finally, for the last example, it suffices to consider $Spin_+(7) \subset Spin(8) \subset Spin(p + 9)$, so that $\m = \so(8) \cap \so_+(7)^\perp = \s_+ = V_L$. We consider an orthonormal set of $\R^{p + 9} = \O \oplus \R^{p+1}$ as
\[
e_1 := 1, e_2 := k, e_3 := i, e_4 := j, \mbox{ and } e_{4+r} = \eps e_r \mbox{ for } r = 1, \ldots, 4, \mbox{ and } e_9, \ldots, e_{12} \in \R^{p+1}.
\]

Note that $[L_i, L_j]|_\H = 2 L_k$, whereas $[L_i, L_j]|_{\H \eps} = -2 L_k$. Thus, in the notation from above, we have $[L_i, L_j] = 2 (E_{12} + E_{34} + E_{56} + E_{78})$. Moreover, $L_i e_{2r - 1} = L_j e_{2r}$ so that $[L_i, E_{2r - 1, s}] = [L_j, E_{2r, s}]$ for $r = 1, \ldots, 4$ and $s \geq 9$. Thus, for the elements $X, Y \in \m \oplus \s$ given as
\[
X = L_i + \sqrt 2 \sum_{r=1}^4 E_{2r, 8 + r} \mbox{ and } Y = L_j + \sqrt 2  \sum_{r=1}^4 E_{2r-1, 8 + r},
\]
one verifies that $[X, Y] = 0$ whereas $X_\m = L_i$ and $Y_\m = L_j$ are linearly independent which contradicts (\ref{eq:condition}).
\end{proof}

\section{Bundles with normal homogeneous collar} \label{sec:restriction}

In this final section, we derive consequences of Theorem~\ref{thm:necessary}, and give a partial classification of the triples $H \subset K \subset G$ for which condition (\ref{eq:condition}) can hold.

We let $\k_0 := \lb \m \rb \lhd \k$ be the ideal generated by $\m$ and let $\h_0 := \h \cap \k_0$, so that $K_0/H_0$ is an almost effective sphere where $H_0 \subset H$ and $K_0 \subset K$ are the connected normal subgroups with Lie algebra $\h_0$ and $\k_0$, respectively. If we let $\h' \lhd \h$ be the ineffective kernel of this action, then we obtain the $Q$-orthogonal decompositions

\be \label{eq:splitting-m}
\h = \h' \oplus \h_0,\mbox{\hspace{1cm}}Ê\k_0 = \h_0 \oplus \m,\mbox{\hspace{1cm}} \k = \h' \oplus \k_0. 
\ee

The almost effective transitive actions on spheres have been classified (\cite{Montgomery}). For each of these actions, we pick a subgroup $H_1 \subset K_0$ which contains $H_0$. Namely, for the homogeneous sphere $S^{15} = Spin(9)/Spin(7)$, we let $H_1 := Spin(8)$, whereas in all other cases, we let $H_1 := (Norm_{K_0} H_0)_0$ be the identity component of the normailzer of $H_0 \subset K_0$. We denote the corresponding $Q$-orthogonal decomposition by

\be \label{eq:splitting-m1}
\k_0 = \h_1 \oplus \m_1.
\ee

From the classification in \cite{Montgomery} it follows that either $H_0 = 1$ and $H_1 = K_0$, so that $\m_1 = 0$, which happens only for $S^1 = U(1)/1$ or $S^3 = SU(2)/1$, or the representation of $H_1$ on $\m_1 := \h_1^\perp$ is irreducible and there is no irreducible $Ad_H$-module in $\m$ which is equivalent to $\m_1 \subset \m$, so that we can choose this particular space for condition (\ref{eq:condition}) in Theorem~\ref{thm:necessary}.

We shall from now on assume that $\m_1 \neq 0$, thus assuming that $\dim H_0 > 0$ and hence, $\dim (K/H) \geq 2$. Moreover, we let $\l := \lb \m_1 \rb \subset \k_0$ be the Lie algebra generated by $\m_1$, and let $L \subset K_0$ be the corresponding connected subgroup. In Table~1, we now list all these groups which follow from the classification in \cite{Montgomery}.

\begin{table}
{\footnotesize
\begin{center}
\begin{tabular}{|c|c|c|c|c|c|c|}
\multicolumn{7}{c}{{\bf Table~1}}\\
\multicolumn{7}{c}{{\sc Almost effective transitive actions of connected Lie groups on spheres}}\\
\multicolumn{7}{c}{{\sc $S^n = K_0/H_0$ with $\dim H_0 >  0$}}
\vspace{2 mm}\\
\hline
&&&&&&\vspace{-3mm}\\
& $\dim S^n$ & $K_0$ & $H_0$ & $H_1$ & $\m_1$ & $L$\\
&&&&&&\vspace{-3mm}\\
\hline
&&&&&&\vspace{-3mm}\\
1 & $n \geq 2$ & $\ba{c}SO(n+1) \\ Spin(n+1)\ea$ & $\ba{c} SO(n) \\ Spin(n)\ea$ & $\ba{c} SO(n) \\ Spin(n)\ea$ & $\R^n$ & $\ba{c} SO(n+1) \\ Spin(n+1)\ea$ \\
&&&&&&\vspace{-3mm}\\
\hline
&&&&&&\vspace{-3mm}\\
2 & $\ba{c}2m+1\\ m \geq 2\ea$ & $\ba{c}T \cdot SU(m+1), \\ T \subset S^1 \mbox{ connected}\ea$ &  $T \cdot SU(m)$ & $T  \cdot U(m)$ & $\C^m$ & $SU(m+1)$\\
&&&&&&\vspace{-3mm} \\
\hline
&&&&&&\vspace{-3mm}\\
3 & $\ba{c}4m+3\\m \geq 1\ea$ & $\ba{c}T \cdot Sp(m+1), \\ T \subset Sp(1) \mbox{ connected}\ea$ & $T \cdot Sp(m)$ & $T \cdot Sp(1) \cdot Sp(m)$ & $\H^m$ & $Sp(m + 1)$\\
&&&&&&\vspace{-3mm}\\
\hline
&&&&&&\vspace{-3mm}\\
7 & $6$ & $G_2$ & $SU(3)$ & $SU(3)$ & $\C^3$ & $G_2$\\
&&&&&&\vspace{-3mm}\\
\hline
&&&&&&\vspace{-3mm}\\
8 & $7$ & $Spin(7)$ & $G_2$ & $G_2$ & $\R^7$ & $Spin(7)$\\
&&&&&&\vspace{-3mm}\\
\hline
&&&&&&\vspace{-3mm}\\
9 & $15$ & $Spin(9)$ & $Spin(7)$ & $Spin(8)$ & $\R^8$ & $Spin(9)$\\
\hline
\multicolumn{7}{l}{{The representations of $H_1$ on $\m_1$ are the standard irreducible representations in each case.}}
\end{tabular}
\end{center}}
\end{table}

\begin{prop} \label{prop:stabilizer}
Let $H \subset K \subset G$ and $\h \subset \k \subset \g$ and $\m_1 \subset \m$ be as above such that (\ref{eq:condition}) is satisfied. Suppose there are elements $0 \neq X_\m \in \m_1$ and $0 \neq Y_\s \in \s$ such that $[X_\m, Y_\s] = 0$. Let $N \subset K_0$ be the identity component of the centralizer of $Y_s$.
Then $N \subset K_0$ acts transitively on $K_0/H_1$.
\end{prop}

\begin{proof}
Evidently $N$ is closed and hence compact. Thus, $N$ acts transitively on $K_0/H_1$ if and only if $\dim(N/(N \cap H_1)) = \dim (K_0/H_1)$, since the former is the dimension of the $N$-orbit of $eH_1 \in K_0/H_1$. Now this equation is equivalent to saying that the projection of $\n\subset \k_0$ to $\m_1$ w.r.t. the splitting (\ref{eq:splitting-m1}) is surjective, or, equivalently, that $\n^\perp \cap \m_1 = 0$, where $\n = \z(Y_\s) \cap \k_0$ is the Lie algebra of $N$.

Observe that $\n^\perp = \z(Y_\s)^\perp + \k_0^\perp = [Y_\s, \g] + \s + \h'$. Thus, what we need to show is the following:

\be \label{condition}
\mbox{If for some $A \in \g$ we have $[Y_\s, A] \in \m_1 \oplus \s \oplus \h'$, then $[Y_\s, A] \in \s \oplus \h'$.}
\ee

Note that $[Y_\s, \k] \subset \s$, so that it suffices to show (\ref{condition}) for all $A \in \s$. Suppose therefore that for some $A \in \s$ we have $[Y_\s, A] \in \m_1 \oplus \s \oplus \h'$.

Then, for $c \in \R$ we let $X := X_\m + c [X_\m, A]$ and $Y := Y_\s + c [Y_\s, A]_{\m_1 \oplus \s}$. Since $[X_\m, A] \in [\k, \s] \subset \s$, the $\m_1$-component of $X$ is indeed $X_\m$, and $X, Y \in \m_1 \oplus \s$. Also, the $\m_1$-component of $Y$ equals the $\m_1$-component of $c [Y_\s, A]$. Moreover, $Q(X_\m, [Y_\s, A]) = Q([X_\m, Y_\s], A) = 0$  since $[X_\m, Y_\s] = 0$, so that
\[
|X_\m \w Y_\m|_Q = |c|\ |X_\m \w [Y_\s, A]_{\m_1}|_Q = |c|\ |X_\m|_Q\ |[Y_\s, A]_{\m_1}|_Q.
\]
On the other hand,
\[
\ba{lll}
[X, Y] & = & c ([X_\m, [Y_\s, A]_{\m_1 \oplus \s}] - \underbrace{[Y_\s, [X_\m, A]]}_{= [X_\m, [Y_\s, A]]}) + c^2 [[X_\m, A], [Y_\s, A]_{\m_1 \oplus \s}]\\ \\
& = & -c\ [X_\m, [Y_\s, A]_{\h'}] + c^2 [[X_\m, A], [Y_\s, A]_{\m_1 \oplus \s}]\\ \\
& = & c^2 [[X_\m, A], [Y_\s, A]_{\m_1 \oplus \s}]
\ea
\]
where the last equation follows since $[X_\m, [Y_\s, A]_{\h'}] \in [\k_0, \h'] = 0$. Thus, by (\ref{eq:condition}) we conclude that there is a $C > 0$ such that for all $c \in \R$,
\[
|c|\ |X_\m|_Q\ |[Y_\s, A]_{\m_1}|_Q \leq c^2\ C\ |[[X_\m, A], [Y_\s, A]_{\m_1 \oplus \s}]|.
\]
Dividing by $|c|$ and taking the limit for $c \ra 0$, we conclude that $[Y_\s, A]_{\m_1} = 0$, i.e., $[Y_\s, A] \in \s \oplus \h'$ which shows (\ref{condition}).
\end{proof}

Let $H \subset K \subset G$ be as in Proposition~\ref{prop:stabilizer} with $\dim H_0 > 0$. Let $L \subset K_0$ be the connected normal subgroup with Lie algebra $\l = \lb \m_1 \rb \lhd \k_0$ from Table~1. We fix the following $Q$-orthogonal $Ad_L$-invariant decomposition:

\be \label{eq:decompose-g}
 \g = \l \oplus \z(\l) \oplus \bigoplus_{\al \in \Phi_1} V_\al \oplus \bigoplus_{\al \in \Phi_2} V_\al =: \n(\l) \oplus \s_1 \oplus \s_2,
 \ee
where $\z(\l)$ and $\n(\l)$ denote the centralizer and the normalizer of $\l$, respectively, and $V_\al$ are non-trivial $Ad_L$-irreducible subspaces for $\al \in \Phi := \Phi_1 \cup \Phi_2$, where
 \[ \ba{ll}
\al \in \Phi_1 & \mbox{if there are elements $0 \neq X_\m \in \m_1$ and $0 \neq Y_s \in V_\al \subset \s$ such that $[X_\m, Y_\s] = 0$},\\Ê\\
\al \in \Phi_2 & \mbox{otherwise}.
\ea
\]

\begin{prop} \label{prop:stabilizer2}
Let $H \subset K \subset G$ and $L \subset K_0$ be as above, i.e., $K_0/H_0$ is one of the entries in Table~1 and (\ref{eq:condition}) holds, and consider the decomposition of $\g$ from (\ref{eq:decompose-g}). Let $\al \in \Phi_1$. Then one of the following holds.

\bi
\item
$L = K_0 = SU(2) \cdot SU(2)'$ so that $\k_0 = \so(4) \cong \su(2) \oplus \su(2)'$, and $V_\al$ is odd dimensional. Furthermore, after permuting $\su(2)$ and $\su(2)'$ if necessary, we have $[\su(2)', \s_1] = 0$.
\item
$K_0/H_0 \cong (T \cdot SU(4))/(T \cdot SU(3))$ so that $L = SU(4)$, and $V_\al$ is the six-dimensional standard representation of $SO(6) = SU(4)/\Z_2$.
\item
$K_0/H_0 \cong Spin(7)/G_2$ so that $L = Spin(7)$, and $V_\al$ is the seven-dimensional standard representation of $SO(7) = Spin(7)/\Z_2$.
\ei
\end{prop}

\begin{proof} Let $N \subset K_0$ be the stabilizer of $Y_\s \in V_\al$ and $\n \subset \k_0$ be its Lie algebra. If $L \subset N$, then $Ad_L(Y_\s) = Y_\s$, so that $Y_\s \in \z(\l)$ which is impossible. Thus, $L \not \subset N$.

By Proposition~\ref{prop:stabilizer}, $N \subset K_0$ acts transitively on $K_0/H_1$. We shall work through the possibilities for $K_0$ from Table~1. 

\bi
\item $K_0/H_1 = SO(n+1)/SO(n) = S^n$ ($Spin(n+1)/Spin(n) = S^n$, resp.) $L = SO(n+1)$ ($L = Spin(n+1)$, resp.).

In this case, $N \subsetneq SO(n+1)$ must be a subgroup which acts transitively on $S^n$, i.e., $N$ must be one of the entries of Table~1. Also, any element $0 \neq X_\m \in \m_1 = \m \subset \so(m+1)$ is a matrix of real rank two, i.e., the Lie algebra $\n \subset \so(n+1)$ of $N \subset SO(n+1)$ must contain such a matrix.

We claim that the only subgroup $N \subsetneq SO(n + 1)$ from Table~1 whose Lie algebra contains elements of real rank two is $U(m) \subset SO(2m)$. Namely, $\su(m) \subset \so(2m)$ contains no elements of complex rank one and hence of real rank two. Next, any element of $\t \oplus \sp(m) \subset \so(4m)$ for $m \geq 2$ is conjugate to an element of the form $X_\m = (\la_0 i, i\ diag(\la_1, \ldots, \la_m))$ with $\la_i \in \R$.  Now the real rank of this element viewed as an endomorphism of $\H^m$ is easily seen to be at least $4$, excluding this case. Next, since $\su(3) \subset \g_2$ have equal rank, any element of $\g_2$ is conjugate to an element of $\su(3) \subset \so(6) \subset \so(7)$. Since $\su(3)$ contains no matrices of real rank two, $\n \cong \g_2$ is also impossible. Likewise, if $X \in \spin(7) \subset \so(8)$ had real rank two, then $X$ would lie in the isotropy algebra of $S^7 = Spin(7)/G_2$, hence $X \in \g_2 \subset \spin(7)$ and thus cannot have rank two by the previous case. Finally, if $X \in \spin(9)$ has rank two, then it lies in the isotropy algebra of $S^{15} = Spin(9)/Spin(7)$, acting on $\R^7 \oplus \R^8$ via the standard and the spin representation, respectively. But by the previous, the action of $X \in \spin(7)$ on $\R^8$ must have rank larger than two which rules out this case as well.

Let $T^m \subset SO(2m)$ be a maximal torus with Lie algebra $\t^m \subset \so(2m)$, and suppose w.l.o.g. that $X_\m \in \t^m$. Let $V_0 \subset \s_1 = \bigoplus_{\al \in \Phi_1} V_\al$ be the subspace stabilized by $T^m$. By our assumption, there are elements of $V_\al$ whose stabilizer is isomorphic to $U(m)$ and hence contains a maximal torus, so that $V_0 \neq 0$. Since $X_\m \in \t^m$, it follows that the stabilizer of any $0 \neq Y \in V_0$ is conjugate $U(m)$ where $T^m \subset U(m) \subset SO(2m)$. But there are only finitely many conjugates of $U(m)$ which contain $T^m$, namely the conjugates by elements of the Weyl group $W := Norm_{SO(2m)} T^m/T^m$, so there are only finitely many choices for this stabilizer. On the other hand, the stabilizer of $0 \neq Y \in V_0$ depends continuously on $Y$, hence all of $V_0$ is stabilized by a fixed subgroup $U(m) \subset SO(2m)$.

Since $V_0$ is invariant under the action of the Weyl group $W$, it follows that $U(m)$ is invariant under conjugation by $W$ as well. In particular $\z(\u(m)) \subset \t^m$ is invariant under the Weyl group, i.e., the Weyl group cannot act irreducibly on $\t^m$, hence $SO(2m)$ cannot be simple, so that $m = 2$, i.e., $N = U(2) \subset SO(4)$.

In particular, $SU(2) \subset N$ acts trivially on $V_0$, and since $SU(2) \subset SO(4)$ is a normal subgroup, it must act trivially on the invariant subspace generated by $V_0$ which is all of $\s_1$. Also, by the above, the representation on $V_\al$ must have $0$ as a weight, hence it is odd dimensional as asserted in the first case.

\item $K_0/H_1 = (T \cdot SU(m+1))/(T \cdot U(m)) = \CP^m$, $m \geq 1$, $L = SU(m+1)$.

If $N \subset T \cdot SU(m+1)$ acts transitively on $\CP^m$, then $S^1 \cdot N \subset U(m+1)$ acts transitively on $S^{2m+1} \subset \C^{m+1}$ and hence must be one of the entries in Table~1. Since $L \not \subset N$, we must have $N = T' \cdot Sp(k) \subset T \cdot SU(2k)$ where $2k = m+1$ and $T' \subset T$ is at most one dimensional. That is, $N \cap L = Sp(k)$, and $k \geq 2$ because $L \not \subset N$.

Let $\t \subset \su(2k)$ be the Lie algebra of the maximal torus consisting of all diagonal matrices. Note that any element $0 \neq X_\m \in \m_1 \subset \su(2k)$ is conjugate to a multiple of $X_0 := diag(i, -i, 0, \ldots, 0) \in \t$. Decompose
\[
V_\al = \bigoplus_\la W_\la
\]
into the weight spaces w.r.t. the maximal torus $\t$.

Since $k \geq 2$, any non-trivial representation $V_\al$ of $L = SU(2k)$ must have a weight $\la_0 \neq 0$ which is annihilated by $X_0 \in \t$ from above, so by our assumption, the stabilizer of each non-zero element of $W_{\la_0}$ must be conjugate to $N \cap L = Sp(k)$ and hence has rank $k$. On the other hand, $W_{\la_0}$ is stabilized by a hyperplane of the maximal torus of $L = SU(2k)$ and hence this stabilizer has rank at least $2k - 2$. It follows that $k = 2$, and we may assume that the Lie algebra of the maximal torus of $\sp(2) \subset \su(4)$ is spanned by $X_0 = diag(i, -i, 0, 0)$ and $diag(0, 0, i, -i)$. That is, any weight $\la$ of $V_\al$ which is annihilated by $X_0 = diag(i, -i, 0, 0) \in \t \subset \su(4)$ must also be annihilated by $diag(0, 0, i, -i)$. From here it easily follows that the only irreducible representation of $L = SU(4)$ with this property is the $6$-dimensional one which is the second case.

\item $K_0/H_1 = (T \cdot Sp(m+1))/(T \cdot Sp(1) \cdot Sp(m)) = \HP^m$, $m \geq 1$, $L = Sp(m+1)$.

If $N \subset T \cdot Sp(m+1)$ acts transitively on $\HP^m$, then $Sp(1) \cdot N \subset Sp(1) \cdot Sp(m+1)$ acts transitively on $S^{4m+3} \subset \H^{m+1}$ and hence must be one of the entries in Table~1. From there is follows that $N = T' \cdot Sp(m+1)$ where $T' \subset T$, so that $L \subset N$ which is impossible.

\item $K_0/H_1 = G_2/SU(3) = S^6$, $L = G_2$.

Table~1 reveals that there is no proper subgroup of $G_2$ acting transitively on $S^6$, so that we must have $N = G_2 = L$, which is again impossible.

\item $K_0/H_1 = Spin(7)/G_2 = S^7$, $L = Spin(7)$.

The only subgroups of $Spin(7)$ which act transitively on $S^7$ are $SU(4) \cong Spin(6)$, $S^1 \cdot Sp(2) \cong Spin(2) \cdot Spin(5)$ and $Sp(2) \cong Spin(5)$.

Let $\t^3 \subset \spin(7)$ be the Lie algebra of a maximal torus $T^3 \subset Spin(7)$, and suppose w.l.o.g. that $X_\m \in \t^3$. If the representation of $Spin(7)$ on $V_\al$ was of spin type, then all elements $0 \neq X_\m$ would act by isomorphisms on $V_\al$, contradicting our assumption on $X_\m$. Thus, $V_\al$ is {\em not }Êof spin type and hence has $0$ as a weight and we let $V_0 \subset \s_1 =\bigoplus_{\al \in \Phi_1} V_\al$ be the space stabilized by $T^3$. Thus, the stabilizer $N$ of any element in $V_0$ must contain $T^3$, hence must be either $Spin(6)$ or $Spin(2) \cdot Spin(5)$. Since none of these two groups is conjugate to a subgroup of the other, we can argue as in the first case to conclude that all $0 \neq Y \in V_0$ have the same subgroup as stabilizer, and this subgroup must be invariant under conjugation by the Weyl group $W$. But $Spin(2) \cdot Spin(5) \subset Spin(7)$ is {\em not} invariant under $W$, hence $N = Spin(6) \subset Spin(7)$. 

It follows that no long root of $\so(7)$ can be a weight of $V_\al$, and the only irreducible representation of $Spin(7)$ which has $0$ but no long root as a weight is the standard one of $SO(7) = Spin(7)/\Z_2$ on $\R^7$ which is the third case.

\item $K_0/H_1 = Spin(9)/Spin(8) \cong SO(9)/SO(8) = S^8$, $L = Spin(9)$.

By Table~1, $SO(9)$ and $Spin(9)$ are the only groups acting transitively on $S^8$, hence we must have $N = Spin(9) = L$ which cannot be the case.
\ei
\vspace{-9mm}
\end{proof}

\begin{prop} \label{prop:stabilizer3}
Let $H \subset K \subset G$ and $L \subset K_0$ be as in Proposition~\ref{prop:stabilizer2}, and consider the decomposition of $\g$ from (\ref{eq:decompose-g}). Suppose that $\al \in \Phi_2$. Then one of the following cases must hold.

\bi
\item
$K_0 = L = Spin(n+1)$ acting on $S^n$ for some $n \geq 2$, and $V_\al$ is an irreducible representation of spin type, i.e., is not the lift of a representation of $SO(n+1) = Spin(n+1)/\Z_2$.
\item
$K_0 = T \cdot SU(2)$ acting on $S^3$, $L = SU(2) = Spin(3)$, and $V_\al$ is an irreducible representation of $L$ of spin type.
\item
$K_0 = T \cdot Sp(2)$ acting on $S^7$, $L = Sp(2) = Spin(5)$, and $V_\al$ is an irreducible representation of $L$ of spin type.
\item
$K_0 = L = Spin(7)$ acting on $S^7$, and $V_\al$ is an irreducible representation of $L$ of spin type.
\item
$K_0 = L = Spin(9)$ acting on $S^{15}$, and $V_\al$ is an irreducible representation of $L$ of spin type.
\ei

In all cases, $L = Spin(n)$ for some $n$, and $\Z_2 \subset Z(Spin(n))$ acts as $\pm Id$ on $V_\al$, where $\Z_2$ is the kernel of the covering $Spin(n) \ra SO(n)$.
\end{prop}

\begin{proof}
Again, we go through the various possibilities for $K_0$ and $L$ from Table~1. First note that the representation of $V_\al$ cannot have $0$ as a weight, since all elements $0 \neq X_\m \in \m_1$ are conjugate to elements of the maximal torus which would annihilate the $0$-weight space.

\bi
\item $L = K_0 = SO(n+1)$ or $Spin(n+1)$, $n \geq 2$.

Any element $X_\m \in \m_1 = \m$ is conjugate to an element of the maximal Lie algebra of $\so(n+1)$ corresponding to the weight element $\th_1$. Every irreducible representation of $Spin(n+1)$ which is not of spin type and does not have $0$ as a weight has all $\th_i$ with $i = 1, \ldots, [(n+1)/2]$ as weights.

If $n \geq 3$, then $\th_2$ is a weight which is annihilated by $X_\m$, contradicting our assumption. If $n = 2$, then any representation of $Spin(3) \cong SU(2)$ which does not have $0$ as a weight is of spin type.

\item $K_0 = T \cdot SU(m+1)$, $m \geq 1$, $L = SU(m+1)$.

Any element $X_\m \in \m_1$ is conjugate to a multiple of $diag(i, -i, 0, \ldots, 0) \in \t$ which is the element of the maximal Lie algebra of $\su(m+1)$ corresponding to the weight element $\th_1 - \th_2$. Every irreducible representation of $SU(m+1)$ has a weight of the form $\la = \th_1 + \ldots + \th_k$ for some $k \leq m$. If $k \geq 2$, then $\lb \th_1 - \th_2, \la \rb = 0$ so that $X$ acts trivially on the weight space $W_\la$ which is impossible. Thus, we must have $k = 1$.

In this case, all elements $\th_i, i = 1, \ldots, m+1$ are weights, and if $m \geq 2$, then $\la = \th_3$ is a weight with $\lb \th_1 - \th_2, \la \rb = 0$, so that $X_\m$ acts trivially on $W_\la$ which is impossible.

Thus, we must have $m = 1$, i.e., $L = SU(2) = Spin(3)$, and since $0$ is not a weight of $V_\al$, it follows that $V_\al$ is even dimensional and therefore of spin type.

\item $K_0 = T \cdot Sp(m+1)$, $m \geq 1$, $L = Sp(m+1)$.

Any element $X_\m \in \m_1$ is conjugate to a multiple of $diag(i, i, 0, \ldots, 0) \in \t$ which is the element of the maximal torus of $\sp(m+1)$ corresponding to the weight element $\th_1 + \th_2$. Every irreducible representation of $Sp(m+1)$ which does not have $0$ as a weight has $\th_i, i = 1, \ldots, m+1$ as weights. Thus, if $m \geq 2$, then $\la = \th_3$ is a weight, and since $\lb \th_1 + \th_2, \la \rb = 0$, it follows that $X$ acts trivially on the weight space $W_\la$ which is impossible. Thus, we must have $m = 1$, hence $L = Sp(2) = Spin(5)$.

Now it is easy to see that a representation of $Spin(5)$ factors through $SO(5)$ if and only if it has $0$ as a weight. Since we assume this {\em not} to be the case, $V_\al$ must be a spin type representation of $L = Sp(2) = Spin(5)$.

\item $K_0 = L = G_2$.

Every representation of $G_2$ has $0$ as a weight, so this case is impossible.

\item $K_0 = L = Spin(7)$ or $K_0 = L = Spin(9)$.

An irreducible representation of $Spin(2k+1)$ factors through a representation of $SO(2k+1) = Spin(2k+1)/\Z_2$ if and only if it has $0$ as a weight. Since $V_\al$ does not have $0$ as a weight, we conclude that it must be a representation of spin type.
\ei

For the final assertion, note that $\Z_2 = ker (Spin(n) \ra SO(n))$ acts non-trivially in all cases, and since $\Z_2 \subset Z(Spin(n))$, it follows from Schur's Lemma that it must act on $V_\al$ as a multiple of the identity.
\end{proof}

\begin{cor} \label{cor:symmetric}
Let $H \subset K \subset G$ and $L \subset K_0 \subset K$ be as in Proposition~\ref{prop:stabilizer2}, and consider the decomposition of $\g$ from (\ref{eq:decompose-g}). Then there is an element $\sigma \in L$ such that $\sigma^2 = 1$, and such that $Ad_\sigma|_{\n(\l) \oplus \s_1} = Id$ and $Ad_\sigma|_{\s_2} = -Id$.

In particular, $(\g, \n(\l) \oplus \s_1) =: (\g, \n_0)$ is a symmetric pair whose reflection is given by $Ad_\sigma$.
\end{cor}

\begin{proof} If $L \not\cong Spin(n)$ for some $n$, then $\s_2 = 0$ by Proposition~\ref{prop:stabilizer3}, so that in this case the claim holds for $\sigma = 1 \in L$.

If $L \cong Spin(n)$ for some $n$, then let $\sigma \in L$ be the non-trivial element in the kernel of the covering $Spin(n) \ra SO(n)$. By Proposition~\ref{prop:stabilizer3}, $Ad_\sigma$ acts as $-Id$ on $\s_2$.

Also, $Ad_\sigma$ acts trivially on the normalizer $\n(\l)$ since $\sigma \in Z(L)$. Finally, if $\s_1 \neq 0$, then by Proposition~\ref{prop:stabilizer2} we have in all cases that $\s_1$ is a representation of $SO(n) = Spin(n)/\Z_2$, so that $Ad_\sigma|_{\s_1} = Id$.
\end{proof}

It thus follows that $\n_0 := \n(\l) \oplus \s_1 \subset \g$ is a Lie subalgebra, and we let $N_0 \subset G$ be the corresponding connected Lie subgroup, so that $G/N_0$ is a symmetric space whose dimension equals that of $\s_2$. Let us consider the various possibilities for $N_0$ is more detail.

\begin{prop} \label{prop:N0}
Let $H \subset K \subset G$ and $L \subset K_0 \subset K$ be as in Proposition~\ref{prop:stabilizer2}, and let $N_0 \subset G$ be the connected subgroup with Lie algebra $\n_0 = \n(\l) \oplus \s_1$. Then there are normal subgroups $\tilde L, \tilde N_0 \subset N_0$ such that $N_0 = \tilde L \cdot \tilde N_0$ and $L \subset \tilde L$, and one of the following holds.

\bi
\item
$\tilde L = L$,
\item
$K_0/H_0 = SU(2) \cdot SU(2)'/\triangle SU(2) = S^3$, $L = K_0$, and $\tilde L = \hat L \cdot SU(2)'$ for some normal subgroup $\hat L \subset \tilde L$, where $SU(2) \subset \hat L$ is such that $Z(SU(2)) \subset Z(\hat L)$.
\item
$K_0/H_0 = (T \cdot SU(4))/(T \cdot SU(3))$, $\dim T \leq 1$, $L = SU(4) \cong Spin(6)$ and $\tilde L = Spin(7)$, where $L = Spin(6) \subset Spin(7)$ by the standard inclusion. Moreover, the adjoint action of $\tilde L$ on $\s_2$ is of spin type.
\item
$K_0/H_0 = Spin(7)/G_2$, $L = K_0 = Spin(7)$ and $\tilde L \in \{SO(8), Spin(8), Spin(9)\}$. Moreover, for each $\tilde L$-irreducible subspace $\tilde V \subset \s_2$, one of the following holds:
\bi
\item $Spin(7) \hookrightarrow Spin(7 + p)$ is the lift of the standard inclusion $SO(7) \subset SO(7+p)$ for $p = 1, 2$, and the action of $\tilde L$ on $\tilde V$ is of spin type, or
\item $Spin(7) \subset SO(8) = \tilde L$ ($Spin(7) \subset Spin(8) = \tilde L$, respectively) is given by (the lift of) the spin representation of $Spin(7)$, and the representation of $\tilde L$ on $\tilde V$ is (the lift of) a representation of $SO(8) \supset Spin(7)$ which does not have $0$ as a weight.
\ei
\ei
\end{prop}

\begin{proof} If $\s_1 = 0$, then $\n_0 = \n(\l)$, hence $\l \lhd \n_0$, so that the first case holds. If $\s_1 \neq 0$, then according to Proposition~\ref{prop:stabilizer2} we have only few possibilities which we shall investigate now.

\bi
\item
$L = SU(2) \cdot SU(2)'$, and $[\su(2)', \s_1] = 0$ so that $\su(2)' \lhd \n_0$ and hence, $N_0 := \hat L \cdot SU(2)' \cdot \tilde N_0$ where $\hat L \subset N_0$ is the normal subgroup generated by $SU(2)$. Since $Z(SU(2))$ acts trivially on all irreducible $V_\al \subset \s_1$ as $V_\al$ is odd dimensional, $Z(SU(2)) \subset Z(\hat L)$ follows.

\item
If $L = SU(4)$, then $V_\al \cong \R^6$ with the standard representation of $SO(6) = SU(4)/\Z_2$. Since there is no $L$-equivariant map $\R^6 \ot \R^6 \ra \R^6$, it follows that $(\n_0, \n(\l))$ is a symmetric pair.

Therefore, we have a decomposition $\n_0 = \g_1 \oplus \ldots \oplus \g_k$ such that $\n(\l) = (\n(\l) \cap \g_1) \oplus \ldots \oplus (\n(\l) \cap \g_k)$, and such that $(\g_i, \g_i \cap \n(\l))$ is an {\em irreducible} symmetric pair. Since $\l \lhd \n(\l)$ is simple, it follows that it must be contained in one of these summands, say, $\l \subset \n(\l) \cap \g_1$. Therefore, if we let $\tilde \l := \g_1$ and $\tilde \n_0 := \g_2 \oplus \ldots \oplus \g_k$ and $\tilde L, \tilde N_0 \subset N_0$ be the corresponding normal subgroups, then $N_0 = \tilde L \cdot \tilde N_0$.

Consider now the irreducible symmetric pair $(\tilde \l, \n(\l) \cap \tilde \l)$. The isotropy group contains $\l = \so(6) = \su(4)$ as an ideal, whose isotropy representation is given by direct sums of the $6$-dimensional representation. From the classification of irreducible symmetric spaces (\cite{Hel}) it follows that $\tilde \l = \so(6 + p)$ and $\n(\l) \cap \tilde \l = \so(6) \oplus \so(p)$ for some $p \geq 1$.

Note that $\h_1 = \u(3) \subset \su(4) \cong \so(6)$ with the standard embedding, and pick a basis of $\R^6$ such that $\z(\u(3)) = \R (E_{12} + E_{34} + E_{56})$ where as before, $E_{rs}$ denotes the skew-symmetric matrix of rank two with $E_{rs} e_r = e_s$ and $E_{rs} e_s= -e_r$, where $(e_r)$ is the standard basis. If $p \geq 2$ so that $6+p \geq 8$, then we let
\[
X_\m := E_{13} - E_{24}, Y_\m := E_{14} + E_{23}, X_\s := \sqrt 2 (E_{17} + E_{38}), Y_\s := -\sqrt 2 (E_{27} + E_{48}).
\]
One easily verifies that $X_\m, Y_\m \in \so(6) \cap \u(3)^\perp = \m_1$, $X_\s, Y_\s \in \s$ and that for $X := X_\m + X_\s$ and $Y := Y_\m + Y_\s$ we have $[X, Y] = 0$, which is impossible according to (\ref{eq:condition}). Thus, we must have $p = 1$ and hence $\tilde \l \cong \so(7)$. Since the corresponding group $\tilde L$ must contain $L = Spin(6)$ as a subgroup, we cannot have $\tilde L \cong SO(7)$, so that $\tilde L \cong Spin(7)$ as claimed.

The representation of $\tilde L = Spin(7)$ on $\s_2$ must be of spin type. For if this was {\em not} the case, then this representation would have $0$ as a weight, hence so would its restriction to $L = Spin(6)$. However, $L$ acts on $\s_2$ of spin type which is a contradiction.

\item
If $L = Spin(7)$, then $V_\al \cong \R^7$ with the standard representation of $SO(7) = Spin(7)/\Z_2$. Arguing as in the previous case, we conclude that there is a normal subgroup $\tilde L \subset N_0$ with Lie algebra $\tilde \l = \so(7 + p)$ for some $p \geq 1$, and the inclusion $L \hookrightarrow \tilde L$ is the lift of the standard inclusion $SO(7) \subset SO(7+p)$. Thus, it follows from Proposition~\ref{prop:octonions} that $p \leq 2$.

If $p = 2$ then $\tilde \l = \so(9)$ hence $\tilde L = Spin(9)$ and in the notation of section~\ref{sec:octonions}, $L = Spin_0(7) \subset Spin(9)$. Since $Z(Spin(9)) = Z(Spin_0(7)) = \Z_2$,  the restriction of a representation of $Spin(9)$ to $Spin_0(7)$ is of spin type if and only if the representation of $Spin(9)$ itself is of spin type showing the claim in this case.

If $p = 1$ then $\tilde L$ is the quotient of $Spin(8)$ by a subgroup of $Z(Spin(8)) = \Z_2 \oplus \Z_2$. Let $\tilde V \subset \s_2$ be a $\tilde L$-irreducible subspace, and let $\Gamma \subset Z(Spin(8))$ be the kernel of the representation of $\tilde L$ on $\tilde V$. Since the restriction of this representation to $Spin(7)$ is of spin type, it follows that $\Gamma \cap Z(Spin(7)) = 0$. Therefore, $\Gamma \subsetneq Z(Spin(8))$, so that this representation cannot have $0$ as a weight. If $\Gamma = 0$, then the representation of $\tilde L = Spin(8)$ on $\s_2$ is of spin type. If $\Gamma \neq 0$, then we must have $\Gamma \cong \Z_2$, and this action is given by a representation of $SO(8) = Spin(8)/\Gamma$ which does not have $0$ as a weight. Moreover, the inclusion $Spin(7) \hookrightarrow SO(8) = Spin(8)/\Gamma$ is the spin representation. \ei
\vspace{-9mm}
\end{proof}

\begin{prop} \label{prop:symmetric-list}
Let $H \subset K \subset G$ and $L \subset K_0 \subset K$ be as in Proposition~\ref{prop:stabilizer2}. As before, let $N_0 \subset G$ be the connected subgroup with Lie algbera $\n(\l) \oplus \s_1$ and let $\tilde L \subset N_0$ be the normal subgroup from Proposition~\ref{prop:N0}. Furthermore, suppose that $L$ is not isomorphic to $SU(2) \cdot SU(2)$.

Then either $\tilde L \subset G$ is a normal subgroup, or $L = \tilde L$ and there is a normal subgroup $G_1 \subset G$ such that $L \subset N_0 \cap G_1$ is normal, and $G_1/(N_0 \cap G_1)$ is one of the entries of  Table~2.
\end{prop}

\begin{table}
{\footnotesize
\begin{center}
\begin{tabular}{|c|c|c|c|}
\multicolumn{4}{c}{{\bf Table~2}}\\
\multicolumn{4}{c}{{\sc List of irreducible symmetric spaces $G_1/(Spin(n) \cdot N')$ with spin type isotropy}}
\vspace{2 mm}\\
\hline
&&&\vspace{-3mm}\\
& $L \cong Spin(n)$ & $G_1/(Spin(n) \cdot N')$ & isotropy representation of $L$\\
&&&\vspace{-3mm}\\
\hline
&&&\vspace{-3mm}\\
1 & $SU(2) \cong Spin(3)$ & $SU(p + 2)/S(U(2) \cdot U(p))$, $p \geq 1$ & $\C^2$ (standard repr.)  \\
&&&\vspace{-3mm}\\
\hline
&&&\vspace{-3mm}\\
2 & $SU(2) \cong Spin(3)$ & $\ba{c}SO(p+4)/(SO(4) \cdot SO(p))$, $p \geq 1\\ SU(2) \hookrightarrow SO(4) \times 1 \ea$ & $\C^2$ (standard repr.)  \\
&&&\vspace{-3mm}\\
\hline
&&&\vspace{-3mm}\\
3 & $Sp(1) \cong Spin(3)$ & $Sp(p + 1)/(Sp(1) \cdot Sp(p))$, $p \geq 1$ & $\C^2$ (standard repr.)  \\
&&&\vspace{-3mm}\\
\hline
&&&\vspace{-3mm}\\
4 & $SU(2) \cong Spin(3)$ & $\ba{c}G_2/SO(4)\\ÊL = SU(2)_1 \subset SO(4)\ea$ & $\C^2$ (standard repr.)  \\
&&&\vspace{-3mm}\\
\hline
&&&\vspace{-3mm}\\
5 & $SU(2) \cong Spin(3)$ & $\ba{c}G_2/SO(4)\\ÊL = SU(2)_3 \subset SO(4)\ea$ & $\C^4 \cong \odot^3(\C^2)$  \\
&&&\vspace{-3mm}\\
\hline
&&&\vspace{-3mm}\\
6 & $SU(2) \cong Spin(3)$ & $F_4/(SU(2) \cdot Sp(3))$ & $\C^2$ (standard repr.)  \\
&&&\vspace{-3mm}\\
\hline
&&&\vspace{-3mm}\\
7 & $SU(2) \cong Spin(3)$ & $E_6/(SU(2) \cdot SU(6))$ & $\C^2$ (standard repr.)  \\
&&&\vspace{-3mm}\\
\hline
&&&\vspace{-3mm}\\
8 & $SU(2) \cong Spin(3)$ & $E_7/(SU(2) \cdot Spin(12))$ & $\C^2$ (standard repr.)  \\
&&&\vspace{-3mm}\\
\hline
&&&\vspace{-3mm}\\
9 & $SU(2) \cong Spin(3)$ & $E_8/(SU(2) \cdot E_7)$ & $\C^2$ (standard repr.)  \\
&&&\vspace{-3mm}\\
\hline
&&&\vspace{-3mm}\\
10 & $Sp(2) \cong Spin(5)$ & $Sp(p + 2)/(Sp(2) \cdot Sp(p))$, $p \geq 1$ & $\H^2$ (standard repr.)  \\
&&&\vspace{-3mm}\\
\hline
&&&\vspace{-3mm}\\
11 & $SU(4) \cong Spin(6)$ & $SU(p + 4)/S(U(4) \cdot U(p))$, $p \geq 1$ & $\Delta_6 \cong \C^4$ (spin repr.)  \\
&&&\vspace{-3mm}\\
\hline
&&&\vspace{-3mm}\\
12 & $Spin(8)$ & $Spin(p + 9)/(Spin(8) \cdot Spin(p+1))$, $p \geq 0$ & $\Delta_8 \cong \R^8$ (spin repr.)  \\
&&&\vspace{-3mm}\\
\hline
&&&\vspace{-3mm}\\
13 & $Spin(9)$ & $F_4/Spin(9)$ & $\Delta_9$ (spin repr.)  \\
&&&\vspace{-3mm}\\
\hline
&&&\vspace{-3mm}\\
14 & $Spin(10)$ & $E_6/(Spin(10) \cdot U(1))$ & $\Delta_{10}^+$ (spin repr.)  \\
&&&\vspace{-3mm}\\
\hline
&&&\vspace{-3mm}\\
15 & $Spin(12)$ & $E_7/(Spin(12) \cdot SU(2))$ & $\Delta_{12}^+$ (spin repr.)  \\
&&&\vspace{-3mm}\\
\hline
&&&\vspace{-3mm}\\
16 & $Spin(16)$ & $E_8/Spin(16)$ & $\Delta_{16}^+$ (spin repr.)  \\
\hline
\end{tabular}
\end{center}}
\end{table}

\begin{proof}
If $\s_2 = 0$ then $G = N_0$ and hence, $\tilde L \subset N_0 = G$ is a normal subgroup by Proposition~\ref{prop:N0}.

Let us suppose that $\s_2 \neq 0$ and $\tilde L = L$. By Proposition~\ref{prop:stabilizer3}, it follows that $L = Spin(n)$ for some $n$, and $L$ acts on $\s_2$ of spin type. Since we excluded the case $n = 4$, $L = \tilde L$ is {\em simple}. Thus, if we decompose the symmetric space $G/N_0$ from Corollary~\ref{cor:symmetric} as
\[
G = G_1 \cdot \ldots \cdot G_k, \mbox{ } N_0 := N_1 \cdot \ldots \cdot N_k, \mbox{ with $N_i = G_i \cap N$},
\]
such that $G_i/N_i$ are {\em irreducible} symmetric spaces, then we may assume w.l.o.g. that $L \subset N_1$ is normal, so that $N_1 = L \cdot N'$ for some normal subgroup $N' \subset N_1$. That is, $G_1/N_1 = G_1/(Spin(n) \cdot N')$ is an irreducible symmetric space such that the restriction of the isotropy representation to $Spin(n)$ is of spin type. From the classification of irreducible symmetric spaces (\cite{Hel}), we conclude that $G_1/N_1$ must be an entry of Table~2.

It remains to exclude the case $\s_2 \neq 0$ and $L \subsetneq \tilde L$. By Proposition~\ref{prop:N0}, it follows again that $\tilde L$ is {\em simple} as we excluded the case $L = SU(2) \cdot SU(2)$, so that as before we may assume that there is a normal subgroup $G_1 \subset G$ such that $G_1/N_1$ is an irreducible symmetric space, where $N_1 = \tilde L \cdot N'$. In particular, it follows that all $\tilde L$-irreducible subspaces of $\tilde V \subset \s_2$ are equivalent.

Let us work through the possibilities for $\tilde L$ given in Proposition~\ref{prop:N0}. If $\tilde L = Spin(n)$ with $n \in \{7,8,9\}$ acts on $\s_2$ of spin type, then $G_1/N_1 = G_1/(Spin(n) \cdot N')$ must be an entry of Table~2, whence the only possibility is $n = 9$ and $G_1 = F_4$. However, the triple $(H_0 \subset K_0 \subset G) = (G_2 \subset Spin_0(7) \subset F_4)$ was shown not to satisfy (\ref{eq:condition}) in Proposition~\ref{prop:octonions}, so that this case is impossible.

Finally, if $\tilde L \in \{SO(8), Spin(8)\}$ acts on $\s_2$ as (the lift of) a representation of $SO(8)$ which does not have $0$ as a weight, then again, the classification of irreducible symmetric spaces implies that $G_1/N_1 = SO(p + 9)/(SO(8) \cdot SO(p+1))$ for some $p \geq 0$, corresponding to the triple $(H_0 \subset K_0 \subset G) = (G_2 \subset Spin_\pm(7) \subset SO(p + 9))$ which was also shown not to satisfy (\ref{eq:condition}) in Proposition~\ref{prop:octonions}, so the proof is completed.
\end{proof}

\begin{prop} \label{prop:list}
Let $H \subset K \subset G$ and $L \subset K_0 \subset K$, and $N_0 \subset G$ be as in Proposition~\ref{prop:symmetric-list} and assume that $L = \tilde L$. Furthermore, let $G_1/(N_0 \cap G_1)$ be the irreducible symmetric space from that proposition. Then one of the following cases holds.
\bi
\item
$L = Sp(1) \subset Sp(p+1)$ and $G_1/(N_0 \cap G_1) = Sp(p+1)/(Sp(1) \cdot Sp(p))$ for some $p \geq 1$,
\item
$L = SU(2)_3 \subset G_2$ and $G_1/(N_0 \cap G_1) = G_2/SO(4)$,
\item
$L = Spin(6) \cong SU(4) \subset SU(5)$, and $G_1/(N_0 \cap G_1) = SU(5)/U(4)$,
\item
$L = Spin(8) \subset Spin(p + 9)$ for $p \in \{0, 1, 2\}$, and $G_1/(N_0 \cap G_1) = Spin(p+9)/(Spin(8) \cdot Spin(p + 1))$, and the action of $K_0 = L$ on $\R^8$ is the spinor representation.
\ei
That is,  of Table~2 only the entries 3, 5, 11 for $p = 1$ and 12 for $p \in \{0, 1, 2\}$ can occur.
\end{prop}

\begin{proof}
In order to show the proposition, we have to exclude all but these possibilities in Table~2.
\bi
\item Entries 1, 2, 4, 6, 7, 8, 9 of Table~2

In this case, $G_1$ is a compact simple Lie group whose Lie algebra does not have type $C_n$ in the classification of Dynkin diagrams, and $L = SU(2)$ is the subgroup generated by a long root $\al$. By our assumption, there is a long root $\beta$ with $\lb \al, \beta \rb = 1$, so that the root spaces of $\al$ and $\beta$ generate a subgroup of $G_1$ which is isomorphic to $SU(3)$ and contains $L = SU(2)$ as subgroup. Since the normalizer of $L \subset SU(3)$ is $U(2)$, it follows that $\s \cap \su(3) \subset \s$ containes the orthogonal complement of $\u(2) \subset \su(3)$. If $\h_0 \subset \l$ is non-trivial, then it is one-dimensional and w.l.o.g. is spanned by $diag(i, -i, 0) \in \su(2) \subset \su(3)$. Therefore, the elements
\be \label{eq:couterexample}
X := \left(\ba{rr|r} 0 & 1 & 1\\ -1 & 0 & 1\\ \hline - 1 & - 1 & 0 \ea \right) \mbox{ and } Y := \left(\ba{rr|r} 0 & i & - i\\ i & 0 & i\\ \hline -i & i & 0 \ea \right)
\ee
are contained in $\m_1 \oplus \s$ and satisfy $[X, Y] = 0$, where \[ X_\m := \left(\ba{rl} 0 & 1\\ -1 & 0 \ea \right), \mbox{ and } Y_\m := \left(\ba{rl} 0 & i\\ i & 0 \ea \right) \]
are independent elements of $\m_1 \subset \l$, contradicting (\ref{eq:condition}).

\item Entry 10 of Table~2: $G_1 = Sp(p+2), L = Sp(2)$.

In this case, $L$ acts either on $S^7$ in which case $\h_0 = \sp(1) \subset \sp(2)$, or on $S^4$ in which case $\h_0 = \sp(1) \oplus \sp(1) \subset \sp(2)$. In either case, for $p = 1$ we can regard the matrices $X, Y$ from (\ref{eq:couterexample}) as elements of $\m_1 \oplus \s \subset \sp(3)$ with $[X, Y] = 0$ and $X_\m \w Y_\m \neq 0$, violating (\ref{eq:condition}). Using the embedding $\sp(3) \hookrightarrow \sp(p+2)$ for $p > 1$ we can rule out this case as well.

\item \label{case-11} Entry 11 of Table~2: $G_1 = SU(p+4), L = SU(4)$.

In this case, $L$ acts either on $S^7$ so that $H_0 = SU(3)$, or it acts on $S^5$ so that $H_0 = Sp(2)$. In the first case, the triple $H_0 \subset K_0 \subset G$ is given as $SU(3) \subset SU(4) \subset SU(p + 4)$. Then we can choose a subgroup $G' \cong SU(3) \subset SU(p + 4)$ such that $H_0 \cap G' = 1$ and $K_0 \cap G' = SU(2)$. Then the elements $X, Y \in \su(3) \cong \g' \subset \g$ from (\ref{eq:couterexample}) show that  (\ref{eq:condition}) is violated.

In the second case, the triple $H_0 \subset K_0 \subset G_1$ is given by $Sp(2) \subset SU(4) \subset SU(p + 4)$. We assert that for $p \geq 2$, $G_1$ contains a subgroup $G' = SU(3) \cdot SU(3)$ such that $K_0 \cap G' = SU(2) \cdot SU(2)$ and $H_0 \cap G' = \triangle SU(2)$. Once this is shown, we can pick the pairs $\hat X := (X, -X), \hat Y := (Y, -Y) \in \g'$ with $X, Y \in \su(3)$ from (\ref{eq:couterexample}), so that $\hat X, \hat Y \in (\h_0 \oplus \z(\h_0))^\perp \subset \m \oplus \s$, satisfy $[\hat X, \hat Y] = 0$ and $\hat X_\m = (X_\m, -X_\m), \hat Y_\m = (Y_\m, -Y_\m)$ are linearly independent, contradicting (\ref{eq:condition}).

To see the existence of $G'$, choose a complex orthonormal basis $\{e_i, f_i, g_r \mid i = 1,2, r = 1, \ldots p\}$ of $\C^{p+4}$ such that $\su(4) \subset \su(p+4)$ is the stabilizer of $span(g_r)$, and $\sp(2) \subset \su(4)$ is the stabilizer of the quaternionic structure $J: span(e_i, f_i) \ra span(e_i, f_i)$ defined as the antilinear map with $J e_i = (-1)^{i+1}e_{i+1}$ and $J f_i = (-1)^{i+1} f_{i+1}$, taking indices mod $2$. If we now define
\[
\g' := \su(3) \oplus \su(3) \subset stab( span(e_1, f_1, g_1) \oplus span(e_2, f_2, g_2)),
\]
then the properties $\su(4) \cap \g' = \su(2) \oplus \su(2)$ and $\sp(2) \cap \g' = \triangle \su(2) \subset \su(2) \oplus \su(2)$ are easily verified.

\item Entry 12 of Table~2: $G_1 = Spin(p+8), L = Spin(8)$.

In this case, the inclusion $H_0 \subset K_0 \subset G_1$ is given by $Spin_\pm(7) \subset Spin(8) \subset Spin(p + 9)$, which violates (\ref{eq:condition}) for $p \geq 3$ by Proposition~\ref{prop:octonions}.

\item Entry 13 of Table~2: $G_1 = F_4, L = Spin(9) \subset F_4$.

In this case, $\m_1 = \so(8)^\perp \cap \so(9)$. If there were no linearly independent $X, Y \in \m_1 \oplus(\s \cap {\frak f}_4)$ with $[X, Y] = 0$, then the normal homogeneous metric on $F_4/Spin(8)$ would have positive curvature. However, by the classification of these spaces from \cite{Ber} this is not the case, so that we must have such elements.

Suppose that for all commuting elements $X, Y \in \m_1 \oplus (\s \cap {\frak f}_4)$ we have $X_\m \w Y_\m = 0$ so that we may assume w.l.o.g. that $Y_\m = 0$. Then $0 = [X, Y] = [X_\m, Y_\s] + [X_\s, Y_\s]$, and since $F_4/Spin(9)$ is a symmetric space, so that $[X_\s, Y_\s] \in \so(9)$, whereas $[X_\m, Y_\s] \in \s \subset \so(9)^\perp$ this implies that $[X_\m, Y_\s] = [X_\s, Y_\s] = 0$. But the non-zero elements of $\m_1 = \so(8)^\perp \subset \so(9)$ act  on $\s$ by isomorphisms, hence $[X_\m, Y_\s] = 0$ and $Y_\s \neq 0$ implies that $X_\m = 0$. Furthermore, $F_4/Spin(9)$ is a symmetric space of rank one, hence $[X_\s, Y_\s] = 0$ implies that $X_\s, Y_\s$ are linearly dependent which is impossible.

Therefore, we must have $[X, Y]$ and $X_\m \w Y_\m \neq 0$ which contradicts (\ref{eq:condition}).

\item Entries 14, 15, 16 of Table~2: $G_1 = E_r, L = K_0 = Spin(2k)$, $H_0 = Spin(2k-1)$, where $(r, k) \in \{ (6, 5), (7, 6), (8, 8)\}$.

We assert that in all three cases, $G_1$ contains a subgroup $G' \cong SU(3) \cdot SU(3)$ such that $SU(2) \cdot SU(2) \subset K_0 \cap G'  \subset U(2) \cdot U(2)$, and $\triangle SU(2) \subset H_0 \cap G' \subset \triangle SU(2) \cdot Z(U(2) \cdot U(2))$. This will be sufficient for our purposes: we choose $\hat X = (X, -X), \hat Y = (Y, -Y) \in \m \oplus \s$ as in case~\ref{case-11} above to derive a contradiction to (\ref{eq:condition}).

In order to see the existence of $G' \subset G_1$, we fix the orthonormal basis $\th_1, \ldots, \th_r$ of the maximal torus of the Lie algebra ${\frak e}_r$ such that
\bi
\item[(i)] the Lie algebra $\so(2k) \subset {\frak e}_r$ has weights $\pm \th_i \pm \th_j$, $1 \leq i < j \leq k$,
\item[(ii)] the weights of the isotropy representation of $Spin(2k)$ are $\frac12 (\eps_1 \th_1 + \ldots + \eps_r \th_k)$, where $\eps_i = \pm 1$ such that $\eps_1 \ldots \eps_k = 1$,
\item[(iii)] $\h_0 = \so(2k-1) \subset \so(2k)$ is the Lie algebra which stabilizes $\th_1$.
\ei
Now we proceed by investigating the three cases separately.

\bi
\item
Suppose $(r, k) = (6, 5)$, corresponding to the Hermitean symmetric space \linebreak $E_6/Spin(10) \cdot U(1)$. Then $(\so(10) \oplus \u(1))^\perp$ is a complex representation of $\so(10) \oplus \u(1)$, and the complex one-dimensional weight spaces $V_\la \subset (\so(10) \oplus \u(1))^\perp$ for $\la = \frac12 (\eps_1 \th_1 + \ldots + \eps_r \th_k)$ are well defined. Moreover, $[V_\la, V_{\la'}] \neq 0$ if and only if $\la - \la'$ is a root of $\so(10)$; in this case, $[V_\la, V_{\la'}] = \g_{\pm(\la - \la')} \oplus \t^1$ for a one dimensional Lie algebra $\t^1 \subset span(\la, \la') \oplus \u(1)$, regarding $\la, \la'$ as elements of the maximal torus of ${\frak e}_r$. We define the weights
\[ 
\la_{1/2} := \frac12 (\pm(\th_1 + \th_2) + \th_3 + \th_4 + \th_5), \mbox{ and } \mu_{1/2} := \frac12 (\pm(\th_1 - \th_2) - \th_3 - \th_4 - \th_5).
\]
Since $\la_i - \mu_j$ is not a root of $\so(10)$, we have $[V_{\la_1} \oplus V_{\la_2}, V_{\mu_1} \oplus V_{\mu_2}] = 0$. Moreover, $\la_1 - \la_2 = \th_1 + \th_2$, and $\mu_1 - \mu_2 = \th_1 - \th_2$, so that the Lie algebras $\g'_1$ and $\g'_2$ generated by $V_{\la_1} \oplus V_{\la_2}$ and $V_{\mu_1} \oplus V_{\mu_2}$, respectively, satisfy $\g'_1 \cap (\so(10) \oplus \u(1)) = \lb \g_{\pm(\th_1 + \th_2)} \rb \oplus \t^1_1$ and  $\g'_2 \cap (\so(10) \oplus \u(1)) = \lb \g_{\pm(\th_1 - \th_2)}\rb \oplus \t^1_2$, i.e., both are isomorphic to $\u(2)$ and act on $V_{\la_1} \oplus V_{\la_2}$ and $V_{\mu_1} \oplus V_{\mu_2}$, respectively, via the standard representation on $\C^2$. Thus, we have the following:

\bi
\item
$(\g_i', \g_i' \cap (\so(10) \oplus \u(1)))$ is a symmetric pair congruent to $(\su(3), \u(2))$,
\item
$[\g_1', \g_2'] = 0$, i.e., $\g' := \g_1' \oplus \g_2' \cong \su(3) \oplus \su(3)$
\item
$\g' \cap (\so(10) \oplus \u(1)) = \so(4) \oplus \u(1)$ so that $\k_0 \cap \g' \cong \so(4)$, which is included into $\so(10)$ by the standard embedding, and into $\g'$ as $\k_0 \cap \g' \cong \su(2) \oplus \su(2) \hookrightarrow \su(3) \oplus \su(3)$.
\ei

Furthermore, since $\h_0 \subset \so(10)$ is the Lie algebra which stabilizes $\th_i$, it follows that $\h_0 \cap \g' \subset \k_0 \cap \g'$ is the standard inclusion $\so(3) \subset \so(4)$ which corresponds to $\triangle \su(2) \subset \su(2) \oplus \su(2)$ as asserted.

\item
Suppose $(r, k) = (7, 6)$, corresponding to the quaternionic symmetric space \linebreak $E_7/Spin(12) \cdot Sp(1)$.
Then $(\so(12) \oplus \sp(1))^\perp$ is a quaternionic representation of $\so(12) \oplus \sp(1)$, and since $-\la$ is a weight whenever $\la$ is, the weight space $W_\la := (V_\la \oplus V_{-\la}) \cap \g \subset (\so(12) \oplus \sp(1))^\perp$ for $\la = \frac12 (\eps_1 \th_1 + \ldots + \eps_r \th_6)$ is well defined as a one-dimensional quaternionic vector space, and $[W_\la, W_{\la'}] \neq 0$ if and only if $\la \pm \la'$ is a root of $\so(12)$; moreover, in this case, $[W_\la, W_{\la'}] = \g_{\la \pm \la'} \oplus \t^1 \oplus \sp(1)$ with $\t^1 = span(\la, \la') \cap (\la \pm \la')^\perp$. We define the Lie algebra $\g'_1 \subset {\frak e}_7$ generated by $W := W_{\la_1} \oplus \ldots \oplus W_{\la_4}$, where
\[ 
\la_i := \frac12 (\eps_1 \th_1 + \eps_2 \th_2 + \eps_3 \th_3 + \th_4 + \th_5 + \th_6), \mbox{ where } \eps_k = \pm 1 \mbox{ and } \eps_1 \eps_2 \eps_3 = 1.
\]
Since $\la_i + \la_j$ is never a root of $\so(12)$, we have $\g_1' \cap (\so(12) \oplus \sp(1)) = [W, W] = \lb \g_{\pm \th_i \pm \th_j}, 1 \leq i < j \leq 3 \rb \oplus \t^1 \oplus \sp(1) \cong \so(6) \oplus \t^1 \oplus \sp(1) \cong \su(4) \oplus \u(2)$ which acts on $W$ via the tensor representation $\C^4 \ot \C^2$. Since $(\g_1', \g_1' \cap (\so(12) \oplus \sp(1)))$ is a symmetric pair, it follows that $\g_1' \cong \su(6)$ with the inclusion $\g_1' \cap (\so(12) \oplus \sp(1)) = \s(\u(4) \oplus \u(2)) \subset \su(6)$ and hence, $\k_0 \cap \g_1' = \u(4) \subset \su(6)$.

Now $\h_0 \cap \g_1' \subset \u(4)$ is the stabilizer of an element of the central extension of the representation of $\su(4) \cong \so(6)$ on $\R^6$. Thus, $\h_0 \cap \g_1' \cong \t^1 \oplus \so(5) \cong \t^1 \oplus \sp(2)$ which is embedded in the canonical way as $\t^1 \oplus \sp(2) \subset \t^1 \oplus \su(4)$.

Thus, we have a subgroup $\g_1' \cong \su(6)$ with $\k_0 \cap \g_1' = \u(4)$ and $\h_0 \cap \g_1' \cong \t^1 \oplus \sp(2) \subset \u(4)$, so that the existence of $\g' \cong \su(3) \oplus \su(3) \subset \g_1'$ with the asserted properties follows as in case~\ref{case-11}.      

\item
Suppose $(r, k) = (8, 8)$, corresponding to the real symmetric space $E_8/Spin(16)$. Then $(\so(16))^\perp$ is a real representation of $\so(16)$, and since $-\la$ is a weight whenever $\la$ is, the weight space $W_\la := (V_\la \oplus V_{-\la}) \cap \g \subset (\so(16))^\perp$ for $\la = \frac12 (\eps_1 \th_1 + \ldots + \eps_r \th_k)$ is well defined as a two-dimensional real vector space, and $[W_\la, W_{\la'}] \neq 0$ if and only if $\la \pm \la'$ is a root of $\so(16)$; moreover, in this case, $[W_\la, W_{\la'}] = \g_{\pm \la \pm \la'} \oplus \t^1$ where $\t^1 \subset span(\la, \la')$. We define the weights
\[
\la_{1/2} := \frac12 (\pm(\th_1 + \th_2) + \th_3 + \ldots + \th_8), \mu_{1/2} := \frac12 (\pm(\th_1 - \th_2) + \th_3 + \th_4 + \th_5 - \th_6 - \th_7 - \th_8).
\]
Since $\la_i \pm \mu_j$ is not a root of $\so(16)$, we have, $[W_{\la_1} \oplus W_{\la_2}, W_{\mu_1} \oplus W_{\mu_2}] = 0$, and the Lie algebras $\g'_1$ and $\g'_2$ generated by $W_{\la_1} \oplus W_{\la_2}$ and $W_{\mu_1} \oplus W_{\mu_2}$, respectively, satisfy $\g'_1 \cap \so(16) = \g_{\pm(\th_1 + \th_2)} \oplus \t^1_1 \cong \u(2)$ and  $\g'_2 \cap \so(16) = \g_{\pm(\th_1 - \th_2)} \oplus \t^1_2 \cong \u(2)$, which act on $W_{\la_1} \oplus W_{\la_2}$ and $W_{\mu_1} \oplus W_{\mu_2}$, respectively, via the standard representation.

Just as in case (a), it now follows that $\g'_i \cong \su(3)$ with $[\g'_1, \g'_2] = 0$, so that for $\g' := \g'_1 \oplus \g'_2$ we have $\k_0 \cap \g' = \u(2) \oplus \u(2)$, and $\h_0  \cap \g' = \triangle \su(2) \oplus \z(\u(2) \oplus \u(2)) \subset \u(2) \oplus \u(2)$ as asserted.
\ei
\ei
\vspace{-9mm}
\end{proof}

\begin{prop} \label{prop:G2}
Let $H \subset K \subset G$ be a triple corresponding to the second entry of Table~3 such that (\ref{eq:condition}) is satisfied.

Then there are compact Lie groups $H' \subset G'$ such that $G = G_2 \times G'$, $K = SO(4) \times H'$ and $H = (SU(2)_1 \cdot T) \times H'$, where $T \subset SU(2)_3$ is at most one dimensional.
\end{prop}

\begin{proof}
Since in this case $\g_2 \lhd \g$, it follows that $G = (G_2 \times G')/\Gamma$ for some finite subgroup $\Gamma \subset Z(G_2 \times G') = Z(G_2) \times Z(G')$. But $Z(G_2) = 1$ so that $\Gamma \subset G'$, and replacing $G'$ by $G'/\Gamma$ we have $G = G_2 \times G'$.

Now $K_0 = SU(2)_3 \subset K$, and $H = T \cdot \tilde H$ where $T \subset K_0$ is at most one dimensional and $\tilde H \subset Z_G SU(2)_3 = SU(2)_1 \times G'$ is the ineffective kernel of the action of $K$ on $S^2$ or $S^3$, respectively. Thus, the proposition follows if we can show that $SU(2)_1 \subset \tilde H \subset H$ or, equivalently, $\su(2)_1 \subset \h$.

As in \cite{ST1}, we decompose the Lie algebra $\g_2$ according to the symmetric pair decomposition of $G_2/SO(4)$ as
\[
\g_2 = (\sp(1)_3 \oplus \sp(1)_1) \oplus \H^2,
\]
where $\sp(1)_3 \subset \sp(2)$ is the Lie algebra spanned by
\[
E_0 := \left(\ba{cc} 3i & \\ & i \ea \right),
E_+:= \left(\ba{cc} 0 & \sqrt 3\\ -\sqrt 3 & 2j \ea \right),
E_-:= \left(\ba{cc} 0 & \sqrt 3 i\\ \sqrt 3 i & 2k \ea \right)
\]
and acts on $\H^2$ from the left, whereas $\sp(1)_1= Im(\H)$ acts via scalar multiplication from the right. Indeed, one verifies the bracket relations
\[
{}[E_0, E_\pm] = \pm 2 E_\mp, \mbox{ and  } [E_+, E_-] = 2 E_0.
\]

Evidently, $\H^2 \subset \s$. Moreover, since $\t \subset \su(2)_3$ is at most one dimensional, we may conjugate $K$ by an appropriate element of $SU(2)_3$ and assume w.l.o.g. that $\t \subset \R E_+$ so that $E_0, E_- \in \m$.

Suppose that $\su(2)_1 \subsetneq \h$. Then there must be an element $s = s_1 + s' \in \s$ with $s' \in \g'$ and $0 \neq s_1 \in \su(2)_1$. Again, after conjugating $K$ by an appropriate element of $SU(2)_1$ and rescaling $s$, we may assume w.l.o.g. that $ad_s: \H^2 \ra \H^2$ corresponds to right multiplication by $i \in Im(\H)$. We let
\[
e_1 := \left( \ba{l} 1\\ 0 \ea \right) \mbox{ and } e_2 := \left( \ba{l} 0\\ 1 \ea \right) \in \H^2
\]
be the standard basis. We assert that $[e_1, e_2] = \la E_+$ for some $0 \neq \la \in \R$. To see this, note that $[e_1, e_2] \in \su(2)_3 \oplus \su(2)_1$ since $(\g_2, \su(2)_3 \oplus \su(2)_1)$ is a symmetric pair. Moreover, for $q \in \su(2)_1 \cong Im(\H)$ we have
\begin{eqnarray*}
Q(q, [e_1, e_2]) = Q([q, e_1], e_2) & = & Q\left(\left( \ba{c} q\\ 0 \ea \right), \left( \ba{c} 0\\ 1 \ea \right)\right) = 0,\\
Q(E_0, [e_1, e_2]) = Q([E_0, e_1], e_2) & = & Q\left(\left( \ba{c} 3i\\ 0 \ea \right), \left( \ba{c} 0\\ 1 \ea \right)\right) = 0,\\
Q(E_-, [e_1, e_2]) = Q([E_-, e_1], e_2) & = & Q\left(\left( \ba{c} 0\\ \sqrt 3 i \ea \right), \left( \ba{c} 0\\ 1 \ea \right)\right) = 0,\\
Q(E_+, [e_1, e_2]) = Q([E_+, e_1], e_2) & = & Q\left(\left( \ba{c} 0\\ -\sqrt 3 \ea \right), \left( \ba{c} 0\\ 1 \ea \right)\right) =  -\sqrt 3 |e_2|^2_Q \neq 0,
\end{eqnarray*}
since $Q|_{\H^2}$ must be a multiple of the standard inner product by irreducibility of $G_2/SO(4)$. Now let us define the following sequence $X_n, Y_n \in \m \oplus \s$.
\[
X_n := E_0 - 3 s - \frac 2{\la n} e_2, \mbox{ and } Y_n := E_- + n e_1
\]
for $0 \neq \la \in \R$ from above. Note that 
\[
{}[E_0 - 3 s, e_1] = E_0 e_1 - 3 e_1 i = 0,
\]
hence
\begin{eqnarray*}
{}[X_n, Y_n] & = & \left[E_0 - 3 s - \frac 2 {\la n} e_2,E_- + n e_1\right] =-2 E_+ + \frac 2 {\la n} E_- e_2 - \frac 2 \la \underbrace{[e_2, e_1]}_{= - \la E_+}\\ 
& = & \frac 2 {\la n} E_- e_2.
\end{eqnarray*}
Thus, $\lim [X_n, Y_n] = 0$ whereas $(X_n)_\m \w (Y_n)_\m = E_0 \w E_- \neq 0$ is constant. This violates (\ref{eq:condition}) and gives the desired contradiction.
\end{proof}

\begin{rem} \label{rem:limit} 
If we consider $G := G_2 \times SU(2)$ and $H = \triangle SU(2) \subset SU(2)_1 \times SU(2) \subset G_2 \times SU(2)$, and $K = H \times SU(2)_3$, then (\ref{eq:condition}) is violated by Proposition~\ref{prop:G2}. On the other hand, one can show that there are no $X, Y \in \m \oplus \s$ with $X_\m \w Y_\m \neq 0$ and $[X, Y] = 0$. That is, condition (\ref{eq:condition}) cannot be weakened to the property that $[X, Y] = 0$ only if $X_\m, Y_\m$ are linearly dependent.
\end{rem}

We call a homogeneous vector bundle $M = G \times_K V$ {\em $G$-reducible} if there is a non-trivial decomposition of the Lie algebra $\g = \g_1 \oplus \g_2$ such that $\k = \k_1 \oplus \k_2$ where $\k_i := \k \cap \g_i$ and such that $\k_2 \subsetneq \g_2$ acts trivially on $V$. Otherwise, we call $M$ {\em $G$-irreducible}.

If $M$ is reducible, then -- allowing for an ineffective action -- we may assume that $G = G_1 \times G_2$. Moreover, after replacing $M$ by a finite $G$-equivariant cover $\tilde M$, we may assume that $K$ is connected and hence $K = K_1 \times K_2$ with $K_i := K \cap G_i$, and $K_2 \subset K$ acts trivially on $V$. Thus, $\tilde M = (G_2/K_2) \times M'$ where $M' = G_1 \times_{K_1} V$ and  $\dim(G_2/K_2) > 0$, and $\tilde M$ has a $G$-invariant metric of nonnegative curvature with normal homogeneous collar if and only if $M'$ does. Thus, it is natural to assume that $M$ is $G$-irreducible.

\

\noindent {\bf Proof of Theorems~\ref{thm:rank-restriction} - ~\ref{thm:rank-4}}
It was already shown in section~\ref{S:constructions} that the bundles asserted in these theorems admit invariant nonnegatively curved metrics with normal homogeneous collar, so it remains to show that there cannot be any others. In particular, we may assume from now on that the rank of the disk bundle is $\geq 3$ and $\neq 4$ so that, in particular, $\m_1 \neq 0$.

If $L \subset G$ is a normal subgroup, then after replacing $G$ by a finite cover, we may assume that $G = G' \times L$. Since $L \subset K$, we must have $K = H' \times L$ for some subgroup $H' \subset G'$. Moreover, $L \subset O(V)$ acts transitively on the unit sphere, so that $M = (G' \times L) \times_{H' \times L} V$ is essentially trivial. 

If $L \subset G$ is {\em not} normal and $L \subsetneq \tilde L$, then, since we assume that rank to be $\neq 4$ and thus, $L \neq SU(2) \cdot SU(2)$, Proposition~\ref{prop:symmetric-list} implies that $\tilde L \subset G$ is normal, whence $G = \tilde L \cdot G'$. Replacing $G$ and hence $\tilde L$ by a finite cover, we may assume that $\tilde L$ is simply connected. Thus, by Proposition~\ref{prop:N0}, we have $L = Spin(n) \subset Spin(m) = \tilde L$ for $(n, m) \in \{(6,7), (7,8),(7,9)\}$, where in either case $L$ acts on $S^7$ so that $M \ra G/K$ is a bundle of rank~$8$.

We have $K = K_0 \cdot H' = Spin(n) \cdot T \cdot H'$ where $\dim T \leq 1$, and $T = S^1$ is possible only for $(n,m) = (6,7)$. Thus, $T \cdot H' \subset Z_{Spin(m)}(Spin(n)) \cdot G'$ where $Z$ denotes the centralizer.

If $(n,m) = (6,7)$ or $(7,8)$, then $(Z_{Spin(m)}Spin(n))_0 = 1$, hence $T \cdot H' \subset G'$. If $T = 1$ then the condition of $G$-irreducibility implies that $G' = 1$, which is the second case for $p = 0$ and the third case of Theorem~\ref{thm:rank-8}, respectively. If $T = S^1$ then $(n,m) = (6,7)$ which corresponds to case 4(c) of that theorem.

If $(n,m) = (7,9)$, then $T = 1$ and $H' \subset Z_{Spin(9)}Spin_0(7) \cdot G' = Spin(2) \cdot G'$. The case $H' = 1$ is the second case of Theorem~\ref{thm:rank-8} with $p = 1$. If $H' \neq 1$, then by $G$-irreducibility, we have $H' \not \subset G'$, corresponding to case 4(b) of Theorem~\ref{thm:rank-8}.

Finally, suppose that $L = \tilde L \subset G$ is not normal and different from $SU(2) \cdot SU(2)$. Then by  Proposition~\ref{prop:list} there must be a normal subgroup $G_1 \subset G$ containing $L$ which is either $SU(5)$, $G_2$, $Sp(p+1)$ for $p \geq 1$, or $Spin(p+9)$ for $p \in \{0,1,2\}$. In either case, $L = Sp(1) \cong SU(2)$ acting on $S^2$ or $S^3$, $L = SU(4) \cong Spin(6)$ acting on $S^5$ or $L = Spin(8)$ acting on $S^7$ by the spin representation which shows that the rank is as asserted in Theorem~\ref{thm:rank-restriction}. Furthermore, $K_0 = L$ and hence, $K = L \cdot H'$ with $H' \subset Norm_{G}L$.

If $G_1 = SU(5)$ so that $G = SU(5) \cdot G'$, then $H' \subset Norm_{G}L = S^1 \cdot G'$. If $H' = 1$ then this is the first case of Theorem~\ref{thm:rank-6}; otherwise, we get the second case of that theorem.

If $G_1 = G_2$, then by Proposition~\ref{prop:G2}, we must have $SU(2)_1 \subset H$, and from there, the hypothesis of $G$-irreducibility implies that $G = G_2$, which yields the first entry in Theorem~\ref{thm:rank-3}.

If $G_1 = Sp(p+1)$ so that $G = Sp(p+1) \cdot G'$, then $H' \subset Norm_{G}L = Sp(p) \cdot G'$ which corresponds to the second entry in Theorem~\ref{thm:rank-3}.

Finally, in the last case, $G_1 = Spin(p + 9)$ so that $G = Spin(p + 9) \cdot G'$ and $K = Spin(8) \cdot H'$ with $H' \subset Z_G Spin(8) = Spin(p+1) \cdot G'$. If $H' = 1$ then this corresponds to the first case of Theorem~\ref{thm:rank-8}; the general case is listed in 4(a) of that theorem.
{\hfill \rule{.5em}{1em}\mbox{}\bigskip}

{\small

\noindent
{\sc Fakult\"at f\"ur Mathematik, Technische Universit\"at Dortmund, Vogelpothsweg 87, 44227 Dortmund, Germany}

\noindent
{\em E-mail address:} lschwach@math.uni-dortmund.de 

\

\noindent
{\sc Department of Mathematics, Saint Joseph University, 5600 City Avenue Philadelphia, PA 19131, USA}

\noindent 
{\em E-mail address:} ktapp@sju.edu}
\end{document}